\documentclass[10pt,reqno,final]{article}

\usepackage{amsmath,amsfonts,amssymb,amsthm,version}
\usepackage{mathrsfs,fancybox,pifont}
\usepackage{graphicx}
\usepackage{url,hyperref}
\usepackage[notcite,notref]{showkeys}
\usepackage{color}
\usepackage{subfigure,multirow}
\usepackage{epstopdf}
\usepackage{cases}
\usepackage{mathtools}
\usepackage{algorithm,algorithmic}
\usepackage{authblk}
\usepackage{extarrows}
\usepackage{lipsum}

\allowdisplaybreaks

\theoremstyle{plain}
\newtheorem{lemma}{Lemma}
\newtheorem{theorem}{Theorem}

\newtheorem{remark}{Remark}

\newtheorem{assumption}{Assumption}

\theoremstyle{definition}

\title{A fast front-tracking approach and its analysis for a temporal multiscale flow problem with a fractional-order boundary growth}
\author{Zhaoyang Wang\thanks{Department of Applied Mathematics, University of Science and Technology Beijing, Beijing 100083, China (zhaoyang584520@163.com)} \, 
Ping Lin\thanks{Corresponding author. Division of Mathematics, University of Dundee, Dundee DD1 4HN, United Kingdom (p.lin@dundee.ac.uk)} \, and
Lei Zhang\thanks{Institute of Natural Sciences, School of Mathematical Sciences, and MOE-LSC, Shanghai Jiao Tong University, Shanghai, China (lzhang2012@sjtu.edu.cn)}}

\affil{}

\date{\today}

\begin{document}
\maketitle

\begin{abstract}
This paper is concerned with a blood flow problem coupled with a slow plaque growth at the artery wall. In the model, the micro (fast) system is the Navier-Stokes equation with a periodically applied force and the macro (slow) system is a fractional reaction equation, which is used to describe the plaque growth with memory effect. We construct an auxiliary temporal periodic problem and an effective time-average equation to approximate the original problem and analyze the approximation error of the corresponding linearized PDE (Stokes) system, where the simple front-tracking technique is used to update the slow moving boundary. An effective multiscale method is then designed based on the approximate problem and the front tracking framework. We also present a temporal finite difference scheme with a spatial continuous finite element method and analyze its temporal discrete error. Furthermore, a fast iterative procedure is designed to find the initial value of the temporal periodic problem and its convergence is analyzed as well. Our designed front-tracking framework and the iterative procedure for solving the temporal periodic problem make it easy to implement the multiscale method on existing PDE solving software. The numerical method is implemented by a combination of the finite element platform COMSOL Multiphysics and the mainstream software MATLAB, which significantly reduce the programming effort and easily handle the fluid-structure interaction, especially moving boundaries with more complex geometries. We present two numerical examples of ODEs and 2-D Navier-Stokes system to demonstrate the effectiveness of the multiscale method. Finally, we have a numerical experiment on the plaque growth problem and discuss the physical implication of the fractional order parameter.


\medskip
\end{abstract}

\section{Introduction}
Multiscale problems have been extensively studied in the past two decades. For spatial multiscale problems, people have developed computable models such as quasi-continuum or atomistic-to-continuum coupling (AtC) models and QM (quantum mechanics) and MM (molecular mechanics) coupling to simulate material behaviors \cite{Tadmor1996, Lin2003, Shapeev2011consistent, Luskin2013, Ortner2012, Wang2021}. For temporal multiscale problems, an example is chemical reactions with concentrations of the species varying from seconds to hours while the time scale of the oscillations of the chemical bonds is in the order of femtoseconds \cite{Ariel2009}. For a general introduction to multiscale methods, we refer to \cite{Weinan2011, Engquist2009, Tadmor2011modeling}.

A common challenge in simulating these multiscale problems is the enormous computational cost when the microscale feature needs to be resolved and corresponding microscopic discretization is performed, or alternatively, the loss of microscopic information if the macroscopic discretization is performed. Macroscopic and microscopic processes should be properly coupled in order to solve such problems effectively and accurately.

The heterogeneous multiscale method (HMM) is one of the most prominent techniques to deal with multiscale problems, which relies on an efficient coupling between the macroscopic and microscopic models \cite{Weinan2003, Abdulle2012, Engquist2005}. For temporal multiscale problems with time scale separation, the macroscale quantities can be computed from the microscale subproblem, and large time steps can be employed to solve the macro-scale model in order to save the computational cost. HMM for temporal multiscale problems only considers local solutions of the microscopic subproblem, thus the initial condition on the local interval needs to be carefully designed and depend on some prior knowledge of the microscale behaviour of the system.

We shall consider in this paper a temporal multiscale problem of the atherosclerosis with a commonly slow plaque growth along the artery boundaries. Frei and Richter \cite{Frei2020} pioneered the study in this direction and presented a basic model of two-way coupled blood and plaque growth in blood arteries, where its numerical simulation is carried out in the Arbitrary Lagrangian-Eulerian (ALE) framework. The ALE may complicate not only governing equations and the analysis of the fluid structure interaction but possibly also the treatment of more complex boundary growth of the plaque. The analysis of the multiscale method is thus done in \cite{Frei2020} for a largely simplified coupled ODE system from the ALE transformed blood-flow-plaque-growth model. In this paper we propose to track the changing domain directly using the front-tracking approach instead of ALE at each time step, thus governing equations are not changed and the method may be more handily applied to general dynamic growth of the plaque. The front-tracking framework not only simplifies the design and analysis of the multiscale method in its original PDE form but also makes it easy to use existing PDE solving software to implement the developed algorithm.

Furthermore, we adopt a more general plaque growth process containing fractional derivatives where memory effects of the plaque accumulation or evolution may be included. Fractional calculus has been extensively studied in the last two decades, especially in the fields of fluid mechanics \cite{Song2016, Churbanov2016, Li2017} and anomalous diffusion \cite{Li2009, Lin2007, Sandev2019, Zhaoyang2021}. Compared with integer order operator, the fractional order operators have a non-local structure, and are suitable for describing the memory and hereditary properties of many physical processes. For our applications, macrophages in the artery wall take up low density lipoproteins (LDL), which carry cholesterol and triglycerides to the tissues, and are finally transformed into foam cells, which are engorged with lipids \cite{Hahn2009}. In the long term, macrophages and foam cells in the artery wall are influenced by a variety of other cells, and thus the diffusion is most likely to be anomalous \cite{Yang2016, Libby2002}. Therefore, the fractional operator with memory effect may be more suitable than the local integer operator to describe the anomalous diffusion process in the artery wall \cite{Metzler2000}. More realistic plaque growth equations can be found in \cite{Yang2016} and \cite{Yang2017}. 

In this paper, we consider the blood flow problem with the atherosclerosis, where the incompressible Navier-Stokes equation with a time-periodic force is coupled with a fractional plaque growth model. Due to the slow plaque growth a significant long-term computation is necessary to observe the change in the domain and the flow properties. The nonlocal property of the fractional order operator makes such a long-term computation impossible. For an efficient long-term computation, it is necessary to develop a temporal multiscale method. The 2-D fluid structure interaction problem in this work faces the challenge that the computational domain changes with time (plaque grows slowly with time) in a fractional order and that the multiscale error analysis will be significantly more difficult in comparison to a simplified integer order ODE system in \cite{Frei2020}. We shall first simplify the procedure by directly tracking the changing domain at each time step using the front-tracking approach and then formulate an auxilary time-periodic system and an effective time-average equation. Based on the auxilary system, a multiscale method is then developed to deal with the macro (slow) and micro (fast) equations separately and two scale variables interact through the growing boundary so as to reduce the computational cost. A simple finite difference scheme in time and a finite element method with an adaptive mesh near the time-dependent boundary will be used to solve both the original and the multiscale method. We shall also introduce a fast iterative procedure to find the initial value of the time-periodic flow subproblem and analyze its convergence rate. The front-tracking framework, designed discrete schemes and the iterative procedure for solving the temporal periodic problem make it easy to implement the multiscale method with existing finite element software. The multiscale method is then implemented through a combination of COMSOL Multiphysics \cite{Pryor2009, Pepper2017} and MATALB. The numerical framework may be applied to a wide range of problems with a periodic applied force and a slow boundary growth.

\subsection*{Outline}
The rest of this manuscript is organized as follows. In Section \ref{section2}, we describe the mathematical model and make necessary assumptions. In Section \ref{section3}, we derive an time-periodic subproblem of the flow equations to approximate the original problem, and analyze the error of the temporal multiscale system at the continuous level. In Section \ref{section4}, we present a time discretization scheme and implementation details of the temporal multiscale system, and analyze the error of its time-discrete scheme. An iterative method to find the initial value of the time-periodic flow equations is also shown in this Section. In Section \ref{section5}, we demonstrate and validate the accuracy and efficiency of our multiscale method through several numerical examples. The effect of the fractional order parameter on plaque growth is also investigated. Finally, we conclude the paper in Section \ref{section6}.

\subsection*{Notation}
For domain $\varOmega$ and $m\geq0$, we use the standard notation for the Sobolev space $H^{m}(\varOmega)$ and the Banach space $L^{m}(\varOmega)$. We use $(\cdot,\cdot)$ to denote the inner product in $L^2(\varOmega)$. Throughout this paper, the letter $C$ will denote a positive constant, with or without subscript, its value may change in different occasions.

\section{Mathematical model and assumptions}
\label{section2}
\subsection{Model problem}
We consider the model which describes the biochemical processes leading to the growth of plaque in blood arteries, as shown in Figure. \ref{Pic1}. We assume that the plaque growth occurs only at the upper boundary of the blood artery, which is controlled by the concentration variable $u(t)$. The blood flow is modeled as an incompressible Newtonian fluid, which is suitable for the description of large arteries \cite{Quarteroni2004}. 
\begin{figure}[H]
	\centering
	\includegraphics[width=9cm,height=3.2cm]{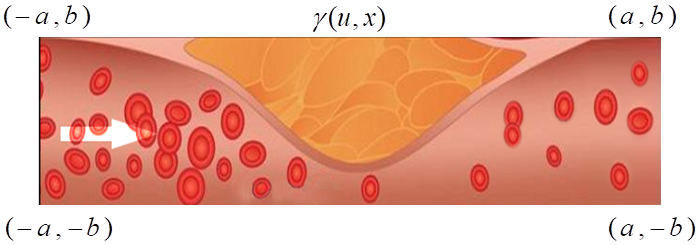}
	\caption{Schematic diagram of atherosclerotic with a plaque growth.}
	\label{Pic1}
\end{figure}

The two-way coupled model of blood flow and plaque growth with fractional derivatives is given as follows
\begin{equation}\label{the1}
	\begin{split}
		&\text{div} \ \textbf{v}=0, \quad \rho(\frac{\partial\textbf{v}}{\partial t}+(\textbf{v}\cdot\nabla)\textbf{v})=\text{div} \sigma(\textbf{v},p)+\textbf{f}, \quad in \ \varOmega(u(t)),\\
		&D_{0^+}^{\alpha}u(t)=\varepsilon R(\textbf{v},u).\\
		&\textbf{v}(0)=\textbf{v}_0, \ u(0)=u_0, \ \varOmega(u(t))=\left\{(x,y):|x|<a, -b<y<b-\gamma(u,x)\right\}.
	\end{split}
\end{equation}

For simplicity of analysis we consider Dirichlet boundary condition, though the method developed here may be applied to other common boundary conditions. In (\ref{the1}), the velocity $\textbf{v}$ and concentration $u$ represent the micro (fast) variable and the macro (slow) variable, respectively. $\rho$ is the density of blood. A periodic force $\textbf{f}(t)=\textbf{f}(t+1)$ is applied to the flow due to the periodic nature of heart pulse. The Cauchy stress tensor $\sigma$ is defined as
\begin{equation}\label{the2}
	\begin{split}
		\sigma=-pI+\rho\nu(\nabla\textbf{v}+\nabla\textbf{v}^T),
	\end{split}
\end{equation}
where $\nu$ is the kinematic viscosity. The function $\gamma(u,x)$ characterizes the shape change of $\Omega$ with respect to time through the concentration $u(t)$.

The reaction term $R\geq 0$ describes the influence of wall shear stress on the boundary growth, which can be simplified as follows \cite{Frei2016, Frei2020}
\begin{equation}\label{the3}
	\begin{split}
		&R=(1+u)^{-1}(1+|\sigma_{WSS}(\textbf{v})|^2)^{-1},\ \sigma_{WSS}(\textbf{v})=\sigma_0^{-1}\int_{\varGamma}\rho\nu(I-\vec{n}\vec{n}^{T})(\nabla\textbf{v}+\nabla\textbf{v}^T)\vec{n}ds,
	\end{split}
\end{equation}
where $\vec{n}$ denotes the outward facing unit normal vector at the deformation boundary $\partial\varOmega$. $\varepsilon\ll 1$ is a small parameter that controls the change of $u$. $D_{0^+}^{\alpha}$ is the Caputo fractional derivative of order $0<\alpha<1$ denoted by \cite{Podlubny1998}
\begin{equation}\label{the4}
	\begin{split}
		D_{0^+}^{\alpha}u(t)=\frac{\partial^\alpha}{\partial t^\alpha}u(t)=\frac{1}{\Gamma(1-\alpha)}\int_0^{t}\frac{u'(s)}{(t-s)^{\alpha}}ds.
	\end{split}
\end{equation}

The Riemann--Liouville (R-L) fractional integral for $\alpha\in (0,1)$ on finite interval $[0,T]$ is defined as
\begin{equation}\label{the5}
	\begin{split}
		I^{\alpha}_{0^+}u(t)=\frac{1}{\Gamma(\alpha)}\int_0^{t}\frac{u(s)}{(t-s)^{1-\alpha}}ds.
	\end{split}
\end{equation}
From the definition of the R-L integral and the Caputo derivative \cite{Podlubny1998}, for $\alpha\in(0,1)$, we can obtain 
\begin{equation}\label{the210}
	D_{0^+}^{\alpha}I^{\alpha}_{0^+}u(t)=u(t), \quad 	I^{\alpha}_{0^+}D_{0^+}^{\alpha}u(t)=u(t)-u(0).
\end{equation}
\begin{equation}
	\label{the260}
	D^{\alpha}_{0^+}D_{0^+}^{\beta}u(t)=D_{0^+}^{\alpha+\beta}u(t), \  I^{\alpha}_{0^+}I_{0^+}^{\beta}u(t)=I_{0^+}^{\alpha+\beta}u(t) \quad (0<\beta<1, 0<\alpha+\beta\leq1).
\end{equation}

\subsection{Assumptions}
Next, we present the essential assumptions which ensure the existence of solutions. 

\begin{assumption}
	\label{assumption1}
	Let $u\in C^1[0,T]$. We assume that the incompressible Navier-Stokes equations on the moving domain $\varOmega(u(t))$ have a solution $\textbf{v}(t)\in H^{2}(\varOmega(t))$ and $p(t)\in H^{1}(\varOmega(t))$. The reaction term is bounded
	\begin{equation}\label{the6}
		\begin{split}
			0\leq R(\textbf{v},u)\leq C_{A1a},
		\end{split}
	\end{equation}
	and has the following Lipschitz condition with respect to slow and fast variables
	\begin{equation}\label{the7}
		\begin{split}
			|R(\textbf{v}_1,u_1)-R(\textbf{v}_2,u_2)|\leq C_{A1b}\left(\Vert \textbf{v}_1-\textbf{v}_2 \Vert_{H^{2}(\varOmega)}+|u_1-u_2|\right).
		\end{split}
	\end{equation}
\end{assumption}

\begin{remark}
	The reaction term $R$ given in (\ref{the3}) satisfies the above assumptions, and its proof can be seen in \cite{Frei2020}.
\end{remark}

\begin{remark}
	We would like to point out that $u(t)$ is bounded based on the properties of fractional operators. Applying the operator $I^{\alpha}_{0^+}$ on both sides of the third equation of (\ref{the1}), we obtain
	\begin{equation}\label{the8}
		\begin{split}
			&|u(t)|=\left|u_0+\frac{\varepsilon}{\Gamma(\alpha)}\int_{0}^{t}(t-s)^{\alpha-1}Rds\right|\leq u_0+\frac{C_{A1a}T^{\alpha}\varepsilon}{\Gamma(\alpha+1)}.
		\end{split}
	\end{equation}
	So to see significant (or $O(1)$) boundary growth we need to compute up to $T=O(\varepsilon^{-\frac{1}{\alpha}})$.
\end{remark}

\begin{assumption} 
	\label{assumption2}
	Let $u\in C^1[0,T]$ be fixed. We assume that the following incompressible Navier-Stokes equation have a unique periodic solution $(\textbf{v}_u,p_u)$
	\begin{equation}\label{the9}
		\begin{split}
			&div \ \textbf{v}_u=0, \ \rho(\frac{\partial\textbf{v}_u}{\partial t}+(\textbf{v}_u\cdot\nabla)\textbf{v}_u)=\text{div} \ \sigma(\textbf{v}_u,p_u)+\textbf{f} \quad in \ [0,1]\times\Omega(u)\\
			&\textbf{v}_u=\textbf{v}_u|_{\varGamma} \ on \  [0,1]\times\partial\Omega(u)\\
			&\textbf{v}_{u}(0)=\textbf{v}_{u}(1) \ in \ \Omega(u).
		\end{split}
	\end{equation}	
	and the solutions are uniformly bounded
	\begin{equation}\label{the10}
		\begin{split}
			\Vert \textbf{v}_u(t)\Vert_{H^{2}(\varOmega)}+\Vert p_u(t)\Vert_{H^{1}(\varOmega)}\leq C_{A2}.
		\end{split}
	\end{equation}
\end{assumption}

\begin{remark}
	For a fixed flow domain, the uniqueness of the periodic solution is guaranteed for moderate Reynolds numbers \cite{Frei2020, Galdi2018}. The periodic Navier-Stokes system can serve as an auxiliary problem, which allows us to quickly solve the temporal multiscale problems.
\end{remark}

We have the following assumption for the time changing shape function $\gamma(u,x)$.
\begin{assumption}
	\label{assumption3} 
	Plaque growth is due to the increased concentration of foam cells \cite{silva2020modeling}, and the plaque growth process is irreversible. We present the following assumptions
	\begin{equation}\label{the201}
		\begin{split}
			\frac{\partial \gamma}{\partial u}\geq 0, \ and \  \frac{\partial \gamma}{\partial t}\geq 0.
		\end{split}
	\end{equation}
\end{assumption}

We remark that in the integer order case $\frac{\partial \gamma}{\partial t}\geq 0$ can be derived from (\ref{the6}) and $\frac{\partial \gamma}{\partial u}$ covers the case considered in \cite{Frei2020}.

\section{Derivation and analysis of the fractional multiscale problem}
\label{section3}

In this section, we derive effective time-average equations for the temporal multiscale system with fractional plaque growth \eqref{the1} based on the assumptions in the previous section. We note that the error analysis between the effective equation and the original equation (\ref{the1}) is performed for the Stokes problem. This is substantially different from the highly simplified system of ODEs \cite{Frei2020}, and can be extended to the full Navier-Stokes system in \eqref{the1}.

\subsection{Derivation of the effective equation}
According to the properties of fractional derivatives, we introduce a new variable
\begin{equation}\label{the11}
	\begin{split}
		\overline{U}(t)=u_0+I^{\alpha}_{0^+}\int_{t}^{t+1}\frac{d^\alpha u(s)}{ds^\alpha}ds,
	\end{split}
\end{equation}
where $I^{\alpha}_{0^+}$ is the R-L fractional integral operator. By inserting $R(\textbf{v}(s),\overline{U}(t))$ in (\ref{the11}), we have
\begin{equation}\label{the12}
	\begin{split}
		&\frac{d^{\alpha}\overline{U}(t)}{dt^{\alpha}}=\int_{t}^{t+1}\frac{d^\alpha u(s)}{ds^\alpha}ds=\varepsilon\int_{t}^{t+1}R(\textbf{v}(s),u(s))ds\\
		&=\varepsilon\int_{t}^{t+1}R(\textbf{v}(s),\overline{U}(t))ds+\varepsilon\int_{t}^{t+1}\left(R(\textbf{v}(s),u(s))-R(\textbf{v}(s),\overline{U}(t))\right)ds. 
	\end{split}
\end{equation}

\begin{lemma}
	\label{lemma10}
	Let $u\in C^1[0,T]$. Then
	\begin{equation}\label{the211}
		\begin{split}
			I^{\alpha}_{0^+}\left(\int_{t}^{t+1}D^{\alpha}_{0^+}u(s)ds\right)=&\int_{t}^{t+1}u(s)ds-\int_{0}^{1}u(s)ds\\
			&+\frac{1}{\Gamma(2-\alpha)}I^{\alpha}_{0^+}\left(\int_{0}^{1}\left[(t+1-r)^{1-\alpha}-t^{1-\alpha}\right]u'(r)dr\right).
		\end{split}
	\end{equation}
\end{lemma}
\begin{proof}
By direct calculation, we have
\begin{equation}\label{the212}
\begin{split}
		&\int_{t}^{t+1}D^{\alpha}_{0^+}u(s)ds
		=\int_{0}^{t+1}D^{\alpha}_{0^+}u(s)ds-\int_{0}^{t}D^{\alpha}_{0^+}u(s)ds\\
		&=\int_{0}^{t+1}D^{\alpha}_{0^+}u(s)ds-I_{0^+}^{1}D_{0^+}^{\alpha}u(t)\\
		&=\frac{1}{\Gamma(1-\alpha)}\int_{0}^{t+1}u'(r)dr\int_{r}^{t+1}(s-r)^{-\alpha}ds+\frac{1}{\Gamma(2-\alpha)}t^{1-\alpha}u_0-I_{0^+}^{1-\alpha}u(t)\\
		&=\frac{1}{\Gamma(2-\alpha)}\int_{0}^{t+1}(t+1-r)^{1-\alpha}u'(r)dr+\frac{1}{\Gamma(2-\alpha)}t^{1-\alpha}u_0-I_{0^+}^{1-\alpha}u(t)\\
		&=\frac{1}{\Gamma(2-\alpha)}\int_{1}^{t+1}(t+1-r)^{1-\alpha}u'(r)dr+\frac{1}{\Gamma(2-\alpha)}\int_{0}^{1}(t+1-r)^{1-\alpha}u'(r)dr\\
		&+\frac{1}{\Gamma(2-\alpha)}t^{1-\alpha}u_0-I_{0^+}^{1-\alpha}u(t)\\
		&\overset{r=s+1}{=}\frac{1}{\Gamma(1-\alpha)}\int_0^t (t-s)^{-\alpha}u(s+1)ds-I_{0^+}^{1-\alpha}u(t)-\frac{1}{\Gamma(2-\alpha)}t^{1-\alpha}u(1)\\
		&+\frac{1}{\Gamma(2-\alpha)}\int_{0}^{1}(t+1-r)^{1-\alpha}u'(r)dr+\frac{1}{\Gamma(2-\alpha)}t^{1-\alpha}u_0\\
		&=D^{\alpha}_{0^+}\int_{t}^{t+1}u(s)ds+\frac{1}{\Gamma(2-\alpha)}\int_{0}^{1}\left[(t+1-r)^{1-\alpha}-t^{1-\alpha}\right]u'(r)dr.
\end{split}
\end{equation}
	
Applying the operator $I^{\alpha}_{0^+}$ on both sides of the equation (\ref{the212}), we have
\begin{equation}\label{the213}
\begin{split}
&I^{\alpha}_{0^+}\int_{t}^{t+1}D^{\alpha}_{0^+}u(s)ds=\int_{t}^{t+1}u(s)ds-\int_{0}^{1}u(s)ds\\
&+\frac{1}{\Gamma(2-\alpha)}I^{\alpha}_{0^+}\left(\int_{0}^{1}\left[(t+1-r)^{1-\alpha}-t^{1-\alpha}\right]u'(r)dr\right),
\end{split}
\end{equation}
which completes the proof.
\end{proof}

\begin{lemma}
	\label{lemma11}
	Let $u\in C^1[0,T=O(\varepsilon^{-\frac{1}{\alpha}})]$. It holds
	\begin{equation}\label{the203}
		u'(t)\geq 0, \ and \ \int_{t}^{t+1}u'(s)=O(\varepsilon).
	\end{equation}
	For $\forall \lambda>0$, by taking $\xi=\min\{1,\frac{\lambda e}{\alpha}\}$, we have
	\begin{equation}\label{the204}
		\int_{0}^t e^{-\lambda(t-s)}u'(s)ds\leq \int_{0}^t (t-s)^{-\alpha\xi}u'(s)ds=O(\varepsilon^{\xi}).
	\end{equation}
\end{lemma}

\begin{proof}
	Differentiating $\gamma(u,x)$ with respect to $t$ and using Assumption \ref{assumption3}, it is easy to see $u'(t)\geq 0$. Applying equation (\ref{the210}), we have
	\begin{equation}
		\begin{split}
			\int_{t}^{t+1}u'(s)ds=\frac{1}{\Gamma(\alpha)}\int_{t}^{t+1}(t-s)^{\alpha-1}O(\varepsilon)ds=O(\varepsilon).
		\end{split}
	\end{equation}
	
	Let $g(z)=z^{-\alpha\xi}-e^{-\lambda z}$. A simple computation shows that $\underset{z\in(0,T)}{min}g(z)=g(\frac{\alpha\xi}{\lambda})\geq 0$. Thus, for $t\in[0,O(\varepsilon^{-\frac{1}{\alpha}})]$, we have
	\begin{equation}\label{the205}
		\begin{split}
			&\int_{0}^t e^{-\lambda(t-s)}u'(s)ds\leq \int_{0}^t (t-s)^{-\alpha\xi}u'(s)ds=\int_{0}^t (t-s)^{\alpha(1-\xi)}(t-s)^{-\alpha}u'(s)ds\\
			&\leq t^{\alpha(1-\xi)}\cdot\int_{0}^t (t-s)^{-\alpha}u'(s)ds=O(\varepsilon^{\xi}).
		\end{split}
	\end{equation}
\end{proof}

\begin{lemma}
	\label{lemma1}
	Let $u\in C^1[0,T]$, it holds that
	\begin{equation}\label{the13}
		\begin{split}
			\left|\varepsilon\int_{t}^{t+1}\left(R(\textbf{v}(s),u(s))-R(\textbf{v}(s),\overline{U}(t))\right)ds\right|\leq C_{L33} \varepsilon^2,
		\end{split}
	\end{equation}
	where $C_{L33}$ depends on $\alpha$, and the constants in Assumptions \ref{assumption1}. 
	\begin{proof}
		By using the Lipschitz condition of the reaction term $R$, \textbf{Lemma \ref{lemma10}} and \textbf{Lemma \ref{lemma11}}, we have
		\begin{equation}\label{the15}
			\begin{split}
				&\left|\varepsilon\int_{t}^{t+1}\left(R(\textbf{v}(s),u(s))-R(\textbf{v}(s),\overline{U}(t))\right)ds\right|\leq\varepsilon C_{A1b}\int_{t}^{t+1}\left|u(s)-\overline{U}(t)\right|ds\\
				&\leq \varepsilon C_{A1b}\int_{t}^{t+1}\left|\int_t^{t+1}\left(u(s)-u(r)\right)dr\right|ds+\varepsilon C_{A1b}\left|\int_0^{1}\left(u(s)-u_0\right)ds\right|\\
				&+\frac{\varepsilon C_{A1b}}{\Gamma(2-\alpha)}I^{\alpha}_{0^+}\left(\int_{0}^{1}\left[(t+1-r)^{1-\alpha}-t^{1-\alpha}\right]u'(r)dr\right)\\
				&\leq \varepsilon C_{A1b}\int_{t}^{t+1}\left|\int_t^{t+1}\int_r^{s}u'(q)dqdr\right|ds+\varepsilon C_{A1b}\int_{0}^{1}\left|\int_{0}^{s}u'(r)dr\right|ds\\
				&+\frac{\varepsilon C_{A1b}}{\Gamma(2-\alpha)}I^{\alpha}_{0^+}\left((t+1)^{1-\alpha}-t^{1-\alpha}\right)\cdot\int_0^1 u'(r)dr\\
				&\leq C\varepsilon^{2}+C_{A1b}(t+1-t)\varepsilon^{2}\leq C_{L33}\varepsilon^{2}.
			\end{split}
		\end{equation}
	\end{proof}
\end{lemma}

We thus have the following estimate for the equation \eqref{the11} of $\overline{U}(t)$,
\begin{equation}\label{the16}
	\begin{split}
		\frac{d^{\alpha}\overline{U}(t)}{dt^{\alpha}}=\varepsilon\int_{t}^{t+1}R(\textbf{v}(s),\overline{U}(t))ds+O(\varepsilon^{2}).
	\end{split}
\end{equation}

The discretization of $\overline{U}(t)$ in a macro-time step $T_n\to T_{n+1}=T_n+\triangle T$ is not accurate because it involves the dynamic evolution of $\textbf{v}(T_n)$ to $\textbf{v}(T_{n+1})$ on the fast scale. The local periodicity can be helpful to effectively capture the velocity feature on the fast scale in $[T_n, T_{n+1}]$. We thus introduce an auxiliary periodic problem at the end of this subsection whose solution $\textbf{v}_{U(t)}$ will be used to approximate the fast scale velocity feature. Here $U(t)$ is an approximation of $\overline{U}(t)$ to be defined at the end of the Section.

Next, we approximate the fractional differential equation (\ref{the16}) by inserting\\
$R(\textbf{v}_{\overline{U}(t)}(s),\overline{U}(t))$ with a fixed $\overline{U}(t)$, that is, writing (\ref{the16}) as
\begin{equation}\label{the17}
	\begin{split}
		&\frac{d^{\alpha}\overline{U}(t)}{dt^{\alpha}}=\varepsilon\int_{t}^{t+1}R(\textbf{v}_{\overline{U}(t)}(s),\overline{U}(t))ds\\
		&+\varepsilon\int_{t}^{t+1} \left(R(\textbf{v}(s),\overline{U}(t))-R(\textbf{v}_{\overline{U}(t)}(s),\overline{U}(t))\right)ds+O(\varepsilon^{2}).
	\end{split}
\end{equation}

By using the Lipschitz condition of $R$, we have 
\begin{equation}\label{the18}
	\begin{split}
		&\varepsilon\int_{t}^{t+1}\left|R(\textbf{v}(s),\overline{U}(t))-R(\textbf{v}_{\overline{U}(t)}(s),\overline{U}(t))\right|ds\\
		&\leq \varepsilon C_{A1b} \int_t^{t+1}\Vert\textbf{v}(s)-\textbf{v}_{\overline{U}(t)}(s)\Vert_{H^2(\varOmega)} ds.
	\end{split}
\end{equation}

In the \textbf{Theorem \ref{theorem4}}, we will prove that
\begin{equation}\label{the19}
	\begin{split}
		\int_t^{t+1}\Vert\textbf{v}(s)-\textbf{v}_{\overline{U}(t)}(s)\Vert_{H^2(\varOmega)} ds=O(\varepsilon^{\frac{1}{2}}),
	\end{split}
\end{equation}
and therefore we have
\begin{equation}\label{the20}
	\begin{split}
		\frac{d^{\alpha}\overline{U}(t)}{dt^{\alpha}}=\varepsilon\int_{t}^{t+1}R(\textbf{v}_{\overline{U}(t)}(s),\overline{U}(t))ds+O(\varepsilon^{\frac{3}{2}}).
	\end{split}
\end{equation}

Here we remark that for integer order case ($\alpha=1$), the estimates can be much better and can be done much easier, that is, the estimate of (\ref{the204}), (\ref{the18}) and (\ref{the19}) is $O(\varepsilon)$ and that of (\ref{the20}) is $O(\varepsilon^2)$. See a relevant Remark \ref{remark3.8} later.

It has been shown that blood flow would be approximately periodic with period 1 after a period of time \cite{Serrin1959note, Kyed2012time}. In other words, even if, for a given initial velocity, the blood flow is not immediately time-periodic, it would become periodic after a certain time. We assume that the blood flow has reached a periodic state initially at time zero since the blood flow has already run for a long time in the body. Therefore, this paper considers the initial value flow problem (\ref{the1}) that has already been time periodic of period 1 from the initial time when the boundary is fixed with initial $u(0)=u_0$, in other word, we may assume $\textbf{v}(0)=\textbf{v}_{u(0)}(0)$ (the definition of $\textbf{v}_u(0)$ is given in (\ref{the9})). This assumption will be used later in the analysis of Lemma \ref{lemma3}.

For the temporal multiscale problem of (\ref{the1}), we give the following auxiliary time periodic flow problem based on (\ref{the20}).
\begin{equation}\label{the21}
	\begin{split}
		&\text{div} \ \textbf{v}_U=0, \quad \rho(\frac{\partial\textbf{v}_U}{\partial t}+(\textbf{v}_U\cdot\nabla)\textbf{v}_U)=\text{div} \  \sigma(\textbf{v}_U,p_U)+\textbf{f}, \quad in \ \Omega(U)\\
		&\textbf{v}_U(0)=\textbf{v}_U(1),\quad \textbf{v}_{U}=\textbf{g} \ \text{on} \ \partial\varOmega(U),\\
		&\frac{d^{\alpha}}{dt^{\alpha}}U(t)=\varepsilon \int_{t}^{t+1}R(\textbf{v}_{U(t)}(s),U(t))ds,\quad U(0)=u_0.
	\end{split}
\end{equation}

From this, we give the framework diagram (Figure. \ref{Multiscale}) of the multiscale algorithm.

\begin{figure}[H]
	\centering
	\includegraphics[width=13cm,height=5.5cm]{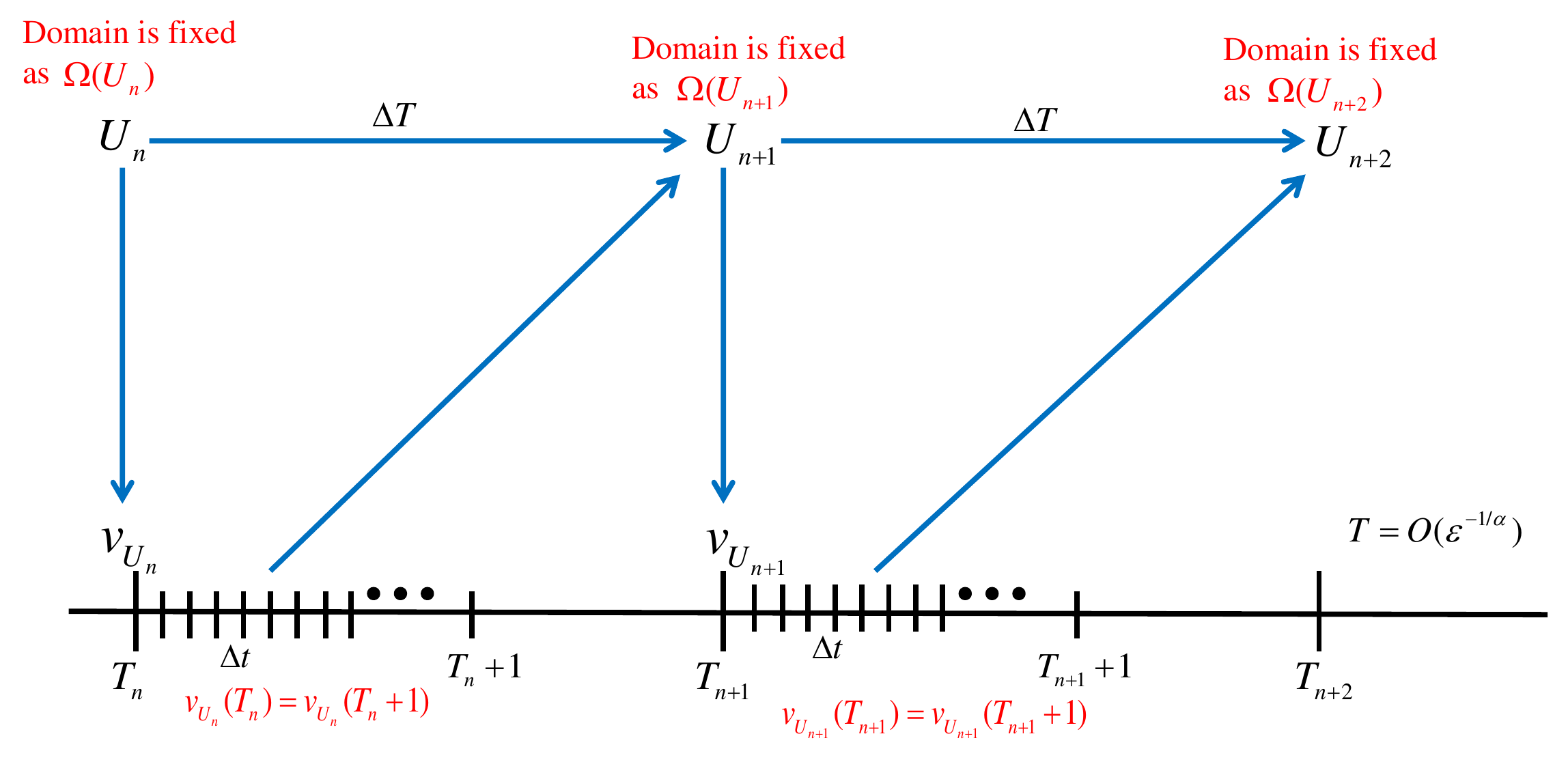}
	\caption{Schematic diagram of the multiscale method. Stepping the macro variable $U(t)$ with macro step size $\Delta T\gg 1$, and then computing the periodic solution $v_{U}(t)$ using the micro step size $\Delta t$ over a unit cycle for a fixed flow domain $\Omega(U)$.}
	\label{Multiscale}
\end{figure}

\subsection{Error analysis of the fractional multiscale problem}
In this section, we will carry out a temporal multiscale error analysis for the Stokes problem. We will show that problem (\ref{the21}) can be used as an effective approximation of problem (\ref{the1}).

\begin{lemma}
	\label{lemma2}
	Let the concentration $u$ be fixed and $0\leq u\leq u_{max}$, $\textbf{v}_u$ is the solution of the corresponding time periodic Stokes problem
	\begin{equation}\label{the22}
		\begin{split}
			&\mathrm{div} \ \textbf{v}_u=0, \quad \rho\frac{\partial\textbf{v}_u}{\partial t}=\rho\nu\Delta \textbf{v}_u-\nabla p_u+\textbf{f}, \quad in \ \varOmega(u)\\
			&\textbf{v}_u(0)=\textbf{v}_u(1), \ \textbf{v}_u|_{\partial\Omega(u)}=0,
		\end{split}
	\end{equation}
	It holds that
	\begin{equation}\label{the23}
		\begin{split}
			\left\Vert \frac{\partial \textbf{v}_u}{\partial u}\right\Vert^2_{L^2(\varOmega(u))}\leq C_{L34a}, \ and \ \int_{0}^{1}\left\Vert \frac{\partial \textbf{v}_u}{\partial u}\right\Vert^2_{H^2(\varOmega(u))}dt\leq C_{L34b}.
		\end{split}
	\end{equation}
	Here $C_{L34a}$ and $C_{L34b}$ depend on $\textbf{f}$, domain $\varOmega$, the constants in Assumptions \ref{assumption2} and \ref{assumption3}.
\end{lemma}

\begin{proof}
	Based on the Stokes equation and Assumption \ref{assumption2}, we have
	\begin{equation}\label{the25}
		\begin{split}
			\left\Vert \frac{\partial \textbf{v}_{u}}{\partial t}\right\Vert_{L^2\left({\varOmega(u)}\right)}\leq C.
		\end{split}
	\end{equation}
	
	For domain $\varOmega(u)$, we can derive from a fixed reference domain $\varOmega(0)$ using the ALE method \cite{Richter2017fluid}
	\begin{equation}\label{the26}
		\begin{split}
			&T: \varOmega(0)\rightarrow\varOmega(u), \quad T=
			\begin{pmatrix}
				x\\
				\frac{b-\gamma(u,x)}{b}y
			\end{pmatrix}
		\end{split}
	\end{equation}
	with Jacobian matrix and determinant given by
	\begin{equation}\label{the27}
		\textbf{F}:=\nabla T=
		\begin{pmatrix}
			1 & 0\\
			\frac{\partial \gamma(u,x)}{\partial x}\frac{b}{y} & \frac{b-\gamma(u,x)}{b}
		\end{pmatrix},\quad
		J:=\det(\textbf{F})=\frac{b-\gamma(u,x)}{b}.
	\end{equation}
	
	The Stokes equations can be mapped to a fixed reference domain
	\begin{equation}\label{the28}
		\begin{split}
			&\nabla \cdot\left(J\textbf{F}^{-1}\textbf{v}_{u}\right)=0, \quad J\frac{\partial \textbf{v}_{u}}{\partial t}=\nabla \cdot \left(J\nabla \textbf{v}_{u}\textbf{F}^{-1}\textbf{F}^{-T}\right)-J\textbf{F}^{-T}\nabla p_{u}+J\textbf{f}, \quad in \ \varOmega(0).
		\end{split}
	\end{equation}
	
	Two different representations of the Stokes (Navier-Stokes) equation in the ALE and in the Eulerian coordinates are equivalent (see \cite{Richter2017fluid}, Chapter 2), it follows that
	\begin{equation}\label{the29}
		\begin{split}
			&C_{*}\left\{\left\Vert \textbf{v}_{u}\right\Vert_{H^2\left({\varOmega(u)}\right)}+\left\Vert p_{u}\right\Vert_{H^1\left({\varOmega(u)}\right)}+\left\Vert \frac{\partial \textbf{v}_{u}}{\partial t}\right\Vert_{L^2\left({\varOmega(u)}\right)}\right\}\\
			&\leq\left\Vert \textbf{v}_{u}\right\Vert_{H^2\left({\varOmega(0)}\right)}+\left\Vert p_{u}\right\Vert_{H^1\left({\varOmega(0)}\right)}+\left\Vert \frac{\partial \textbf{v}_{u}}{\partial t}\right\Vert_{L^2\left({\varOmega(0)}\right)}\\
			&\leq C_{**}\left\{\left\Vert \textbf{v}_{u}\right\Vert_{H^2\left({\varOmega(u)}\right)}+\left\Vert p_{u}\right\Vert_{H^1\left({\varOmega(u)}\right)}+\left\Vert \frac{\partial \textbf{v}_{u}}{\partial t}\right\Vert_{L^2\left({\varOmega(u)}\right)}\right\}
		\end{split}
	\end{equation}
	
	Differentiating the equation (\ref{the28}) with respect to $u$. Let $\textbf{w}_u=\frac{\partial \textbf{v}_u}{\partial u}$, we obtain
	\begin{equation}\label{the30}
		\begin{split}
			&\nabla \cdot\left(J\textbf{F}^{-1}\textbf{w}_{u}\right)+\nabla \cdot\left(\frac{\partial (J\textbf{F}^{-1})}{\partial u}\textbf{v}_{u}\right)=0\\
			& \frac{\partial J}{\partial u}\frac{\partial \textbf{v}_{u}}{\partial t}+J\frac{\partial \textbf{w}_{u}}{\partial t}=\nabla \cdot \left(J\nabla \textbf{w}_{u}\textbf{F}^{-1}\textbf{F}^{-T}\right)+\nabla \cdot \left(\nabla \textbf{v}_{u}\frac{\partial \left(J\textbf{F}^{-1}\textbf{F}^{-T}\right)}{\partial u}\right)\\
			&-\frac{\partial (J\textbf{F}^{-T})}{\partial u}\nabla p_{u}-J\textbf{F}^{-T}\nabla \left(\frac{\partial p_{u}}{\partial u}\right)+\frac{\partial J}{\partial u}\textbf{f}, \quad in \ \varOmega(0).
		\end{split}
	\end{equation}
	
	Let $\textbf{z}_u=\textbf{w}_u+J^{-1}\textbf{F}\frac{\partial (J\textbf{F}^{-1})}{\partial u}\textbf{v}_{u}$, (\ref{the30}) can be rewritten as
	\begin{equation}\label{the31}
		\begin{split}
			&\nabla \cdot\left(J\textbf{F}^{-1}\textbf{z}_{u}\right)=0\\
			& \frac{\partial J}{\partial u}\frac{\partial \textbf{v}_{u}}{\partial t}
			-\textbf{F}\frac{\partial \left(J\textbf{F}^{-1}\right)}{\partial u}\frac{\partial \textbf{v}_{u}}{\partial t}+J\frac{\partial \textbf{z}_{u}}{\partial t}=\nabla \cdot \left(J\nabla \textbf{z}_{u}\textbf{F}^{-1}\textbf{F}^{-T}\right)\\
			&+\nabla \cdot \left(\nabla \textbf{v}_{u}\frac{\partial \left(J\textbf{F}^{-1}\textbf{F}^{-T}\right)}{\partial u}\right)-\nabla \cdot \left(J\nabla \left(\textbf{S}\textbf{v}_{u}\right)\textbf{F}^{-1}\textbf{F}^{-T}\right)\\
			&-\frac{\partial (J\textbf{F}^{-T})}{\partial u}\nabla p_{u}-J\textbf{F}^{-T}\nabla \left(\frac{\partial p_{u}}{\partial u}\right)+\frac{\partial J}{\partial u}\textbf{f}, \quad in \ \varOmega(0)
		\end{split}
	\end{equation}
	with time periodic condition $\textbf{z}_u(0)=\textbf{z}_u(1)$ and homogeneous Dirichlet boundary condition. Here, $S=J^{-1}\textbf{F}\frac{\partial (J\textbf{F}^{-1})}{\partial u}$.
	
	It is easy to check that both the determinant $J$ and the elements associated with the matrix $\textbf{F}$ are non-negative and bounded from below and above,
	\begin{equation}\label{the32}
		\begin{split}
			0< C_{\min}\leq \left\{J,\frac{\partial J}{\partial u},\Vert \textbf{F}\Vert_{L^2},\left\Vert \frac{\partial \textbf{F}}{\partial u}\right\Vert_{L^2},\Vert \textbf{S}\Vert_{L^2}\right\}\leq C_{\max}.
		\end{split}
	\end{equation}
	
	Multiplying the second equation of (\ref{the31}) by $\textbf{z}_u$ and integrating with respect to the space variables on the domain $\varOmega(0)$, we get
	\begin{equation}\label{the33}
		\begin{split}
			& \frac{\partial J}{\partial u}\left(\frac{\partial \textbf{v}_{u}}{\partial t},\textbf{z}_{u}\right)+\frac{1}{2}J\frac{d}{dt}\left\Vert \textbf{z}_{u}\right\Vert^2_{L^2}+\left(J\nabla \textbf{z}_{u}\textbf{F}^{-1},\nabla \textbf{z}_{u}\textbf{F}^{-1}\right)=\left(\textbf{F}\frac{\partial \left(J\textbf{F}^{-1}\right)}{\partial u}\frac{\partial \textbf{v}_{u}}{\partial t},\textbf{z}_{u}\right)\\
			&-\left(\nabla \textbf{v}_{u}\frac{\partial \left(J\textbf{F}^{-1}\textbf{F}^{-T}\right)}{\partial u},\nabla \textbf{z}_{u}\right)+\left(\nabla \left(\textbf{S}\textbf{v}_{u}\right)\textbf{F}^{-1}\textbf{F}^{-T},\nabla \textbf{z}_{u}\right)\\
			&-\left(\frac{\partial (J\textbf{F}^{-T})}{\partial u}\nabla p_{u},\textbf{z}_{u}\right)+\frac{\partial J}{\partial u}\left(\textbf{f},\textbf{z}_{u}\right).
		\end{split}
	\end{equation}
	
	Combining Young's inequality, Poincar\'{e} inequality, (\ref{the29}) and (\ref{the32}), we can derive the following inequalities,
	\begin{equation}\label{the34}
		\begin{split}
			\frac{d}{dt}\left\Vert \textbf{z}_{u}\right\Vert^2_{L^2}+C_1\left\Vert \textbf{z}_{u}\right\Vert^2_{L^2}\leq \frac{d}{dt}\left\Vert \textbf{z}_{u}\right\Vert^2_{L^2}+C_1\gamma\left\Vert \nabla \textbf{z}_{u}\right\Vert^2_{L^2}\leq C_2,
		\end{split}
	\end{equation}
	where $\gamma$ is the Poincar\'{e} constant.
	
	Multiplying $e^{C_{1}t}$ on both sides and integrating in $t$ and using the periodic condition of $\textbf{z}_u$, we obtain
	\begin{equation}\label{the38}
		\begin{split}
			\left\Vert \textbf{z}_{u}(t)\right\Vert^2_{L^2}\leq \frac{C_2}{C_1}.
		\end{split}
	\end{equation}
	
	Integrating (\ref{the34}) in $t$ on [0,1], we obtain
	\begin{equation}\label{the39}
		\begin{split}
			\int_0^1\left\Vert \nabla \textbf{z}_{u}\right\Vert^2_{L^2}dt\leq \frac{C_2}{C_1\gamma}.
		\end{split}
	\end{equation}
	
	Multiplying the second equation of (\ref{the31}) by $-P\left[\nabla \cdot \left(J\nabla \textbf{z}_{u}\textbf{F}^{-1}\textbf{F}^{-T}\right)\right]$, where $P$ is the $L^2$-orthogonal projection from $L^2(\varOmega(0))^2$ to $H$, and 
	\begin{equation}\label{the40}
		\begin{split}
			H=\left\{\textbf{z}_{u};\nabla\cdot(J\textbf{F}^{-1}\textbf{z}_u)=0 \ in \ \varOmega(0) \ and \ (J\textbf{F}^{-1}\textbf{z}_u)\cdot \vec{n}|_{\partial\varOmega(0)}=0\right\}.
		\end{split}
	\end{equation}
	
	Integrating on the domain $\varOmega(0)$, using the properties of the projection (see Section. 2 in \cite{Robinson2016three}) and Young inequality, we have
	\begin{equation}\label{the41}
		\begin{split}
			\frac{J^2}{2}\frac{d}{dt}\left\Vert \left(\nabla \textbf{z}_{u}\textbf{F}^{-1}\right)\right\Vert^2_{L^2}+C_3\left\Vert P\left[\nabla \cdot \left(J\nabla \textbf{z}_{u}\textbf{F}^{-1}\textbf{F}^{-T}\right)\right]\right\Vert^2_{L^2}\leq C_3
		\end{split}
	\end{equation}
	with time periodic condition $\left\Vert \left(\nabla \textbf{z}_{u}\textbf{F}^{-1}\right)\right\Vert_{L^2}(0)=\left\Vert \left(\nabla \textbf{z}_{u}\textbf{F}^{-1}\right)\right\Vert_{L^2}(1)$.
	
	Integrating (\ref{the41}) from 0 to 1 and following the proof of  in \cite[Theorem 2.22]{Robinson2016three}, we obtain
	\begin{equation}\label{the42}
		\begin{split}
			&C_4\int_{0}^{1}\left\Vert \left[\nabla \cdot \left(J\nabla \textbf{z}_{u}\textbf{F}^{-1}\textbf{F}^{-T}\right)\right]\right\Vert^2_{L^2}dt\leq C_3\int_{0}^{1}\left\Vert P\left[\nabla \cdot \left(J\nabla \textbf{z}_{u}\textbf{F}^{-1}\textbf{F}^{-T}\right)\right]\right\Vert^2_{L^2}dt\leq C_5.
		\end{split}
	\end{equation}
	
	Based on the relationship between $\textbf{w}_u$ and $\textbf{z}_u$ and with the help of equation (\ref{the29}), we have
	\begin{equation}\label{the43}
		\begin{split}
			\left\Vert \textbf{w}_{u}(t)\right\Vert^2_{L^2(\varOmega(u))}\leq C_{*}^{-1}\left\Vert \textbf{w}_{u}(t)\right\Vert^2_{L^2(\varOmega(0))}\leq C,
		\end{split}
	\end{equation}
	and 
	\begin{equation}\label{the44}
		\begin{split}
			\int_{0}^{1}\left\Vert \textbf{w}_{u}(t)\right\Vert^2_{H^2(\varOmega(u))}dt\leq C_{*}^{-1}\int_{0}^{1}\left\Vert \textbf{w}_{u}(t)\right\Vert^2_{H^2(\varOmega(0))}dt\leq C.
		\end{split}
	\end{equation}
	
	The proof is completed.
\end{proof}

\begin{remark}
	It is possible to obtain the same results for a nonhomogeneous Dirichlet boundary condition since it can be transformed to the homogeneous Dirichlet boundary condition with appropriate assumptions on the smoothness of the domain and the regularity of the boundary value function.
\end{remark}

The following lemma gives the error estimate between the solution $\textbf{v}(t)$ of the original problem (\ref{the1}) and the time periodic solution $\textbf{v}_{u(t)}(t)$. For a fixed $u(t)$, we have the family of periodic solutions
\begin{equation}\label{the46}
	\begin{split}
		&\text{div} \ \textbf{v}_{u(t)}(s)=0, \quad \rho\frac{\partial}{\partial s}\textbf{v}_{u(t)}(s)=\rho\nu\Delta \textbf{v}_{u(t)}(s)-\nabla p_{u(t)}(s)+\textbf{f}(s), \quad in \ \Omega(u(t))\\
		&\textbf{v}_{u(t)}(0)=\textbf{v}_{u(t)}(1), \ s\in[0,1], \ t\in[0,T]. 
	\end{split}
\end{equation}
We note here that although $\textbf{v}_{u(t)}(s)$ is defined on $[0,1]$, it can be periodically extended to $[0,T]$.

\begin{lemma}
	\label{lemma3}
	Let $u\in C^1[0,T]$, $\textbf{v}(t)$ be the solution of the original Stokes problem and the initial values of $\textbf{v}$ and $\textbf{v}_{u(t)}$ are identical, i.e., $\textbf{v}_{u(0)}(0)=\textbf{v}_0$. For $t\leq O(\varepsilon^{-\frac{1}{\alpha}})$, it holds that
	\begin{equation}\label{the47}
		\begin{split}
			\Vert \textbf{v}_{u(t)}(t)-\textbf{v}(t)\Vert_{L^2}^2\leq C_{L36a}\varepsilon^{\frac{\xi}{2}} \ and \ \int_{t}^{t+1}\Vert \textbf{v}_{u(t)}(t)-\textbf{v}(t)\Vert_{H^2}^2dt\leq C_{L36b}\varepsilon,
		\end{split}
	\end{equation}
	where $C_{L36a}$ and $C_{L36b}$ depend on $C_{L34a}$ and the constants in Assumptions \ref{assumption1}. $\xi=\min\{1,\frac{\nu e}{\alpha\gamma}\}$ and $\gamma$ is the the Poincar\'{e} constant. 
\end{lemma}

\begin{proof}
	For $\textbf{v}_{u(t)}(t)$, by the chain rule, it holds that
	\begin{equation}\label{the49}
		\begin{split}
			\frac{\partial }{\partial t}\textbf{v}_{u(t)}(t)=\frac{\partial }{\partial s}\textbf{v}_{u(t)}(s)\Big|_{s=t}+\frac{\partial \textbf{v}_{u(t)}}{\partial u(t)}(t)\cdot u'(t).
		\end{split}
	\end{equation}
	
	Combining (\ref{the46}), $\textbf{v}_{u(t)}$ satisfies the following PDE
	\begin{equation}\label{the50}
		\begin{split}
			\rho\frac{\partial}{\partial t}\textbf{v}_{u(t)}(t)-\rho\frac{\partial \textbf{v}_{u(t)}}{\partial u(t)}(t)\cdot u'(t)=\rho\nu\Delta \textbf{v}_{u(t)}(t)-\nabla p_{u(t)}+\textbf{f}(t),\quad \textbf{v}_{u(0)}(0)=\textbf{v}_0.
		\end{split}
	\end{equation}
	
	Let $\textbf{w}(t)=\textbf{v}_{u(t)}(t)-\textbf{v}(t)$, we have
	\begin{equation}\label{the51}
		\begin{split}
			&\text{div} \textbf{w}=0, \quad \rho\frac{\partial}{\partial t}\textbf{w}(t)=\rho\frac{\partial \textbf{v}_{u(t)}}{\partial u(t)}(t)\cdot u'(t)+\rho\nu\Delta \textbf{w}(t)-\nabla p_{u(t)}+\nabla p,\\
			& \textbf{w}(0)=0,\quad \textbf{w}|_{\partial\varOmega(u(t))}=0.
		\end{split}
	\end{equation}
	
	Multiplying (\ref{the51}) by $\textbf{w}$ and then integrating on the domain $\varOmega(u(t))$ yields
	\begin{equation}\label{the52}
		\begin{split}
			\frac{1}{2}\int_{\varOmega(u(t))}\frac{d}{d t}\left(\textbf{w}(t)\right)^2+\nu\Vert\nabla \textbf{w}(t)\Vert_{L^2}^2=\left(\frac{\partial \textbf{v}_{u(t)}}{\partial u(t)}(t)\cdot u'(t),\textbf{w}(t)\right).
		\end{split}
	\end{equation}
	
	For the first term, with the help of Assumption \ref{assumption3}, and integrating over the domain $\varOmega(u(t))$, we have
	\begin{equation}\label{the53}
		\begin{split}
			&\frac{d }{d t}\int_{\varOmega(u(t))}\left(\textbf{w}(t)\right)^2=\frac{d }{d t}\int_{-a}^{a}\int_{-b}^{b-\gamma(u,x)}\left(\textbf{w}(t)\right)^2dxdy\\
			&=-u'\int_{-a}^{a}\left(\textbf{w}(t,x,{b-\gamma(u,x)})\right)^2\cdot \frac{\partial \gamma(u,x)}{\partial u}dx
			+\int_{\varOmega(u(t))}\frac{d}{d t}\left(\textbf{w}(t)\right)^2\leq \int_{\varOmega(u(t))}\frac{d}{d t}\left(\textbf{w}(t)\right)^2.
		\end{split}
	\end{equation}
	
	Combining Poincar\'e inequality, Young's inequality, \textbf{Lemma \ref{lemma2}}. Integrating from $0$ to $t$ and using \textbf{Lemma \ref{lemma11}}, we obtain
	\begin{equation}\label{the54}
		\begin{split}
			\Vert \textbf{w}(t)\Vert_{L^2}^2\leq C\int_{0}^{t}e^{-\lambda(t-s)}u'(s)ds=O(\varepsilon^{\xi}), 
		\end{split}
	\end{equation}
	where $\lambda=\frac{\nu}{\gamma}$, $\gamma$ is the the Poincar\'{e} constant and $\xi=\min\{1,\frac{\nu e}{\alpha\gamma}\}$.
	
	Multiplying (\ref{the51}) by $\textbf{w}_t$ and integrating on the domain $\varOmega(u(t))$, we have
	\begin{equation}\label{the56}
		\begin{split}
			\left\Vert \textbf{w}_t\right\Vert_{L^2}^2+\frac{\nu}{2}\frac{d }{dt}\left\Vert \nabla \textbf{w}\right\Vert^2_{L^2}\leq Cu'(t)+\frac{1}{2}\left\Vert \textbf{w}_t\right\Vert_{L^2}^2.
		\end{split}
	\end{equation}
	
	It thus follows that
	\begin{equation}\label{the57}
		\begin{split}
			\left\Vert \nabla \textbf{w}(t+1)\right\Vert^2_{L^2}-\left\Vert \nabla \textbf{w}(t)\right\Vert^2_{L^2}\leq C\varepsilon.
		\end{split}
	\end{equation}
	
	Multiplying (\ref{the51}) by $-P\Delta \textbf{w}$, $P$ is the orthogonal projection from $L^2(\varOmega(u(t)))^2$ to $G$
	\begin{equation}\label{the58}
		\begin{split}
			G=\left\{\textbf{w}; \text{div}\ \textbf{w}=0 \ \text{in} \ \varOmega(u(t)) \ \text{and} \ \textbf{w}\cdot\vec{n}|_{\partial\varOmega(u(t))=0} \right\}.
		\end{split}
	\end{equation}
	
	Integrating on the domain $\varOmega(u(t))$, we obtain
	\begin{equation}\label{the59}
		\begin{split}
			\frac{1}{2}\frac{d}{dt}\left\Vert \nabla \textbf{w}\right\Vert_{L^2}^2+C\left\Vert \textbf{w}\right\Vert_{H^2}^2\leq Cu'(t).
		\end{split}
	\end{equation}
	
	The proof concludes with integrating (\ref{the59}) in $t$ on $[t,t+1]$ and using (\ref{the57}), 
	\begin{equation}\label{the60}
		\begin{split}
			\int_{t}^{t+1}\Vert\textbf{w}(t)\Vert_{H^2}^2dt\leq C\varepsilon.
		\end{split}
	\end{equation}
\end{proof}

Next, we prove a theorem to show the error estimation between the auxiliary time periodic problem (\ref{the21}) and the original problem (\ref{the1}).

\begin{theorem}
	\label{theorem4}
	Let $u\in C^1[0,T]$. $(\textbf{v},u)$ be defined as the Stokes problem corresponding to (\ref{the1}) and $(\textbf{v}_{U(t)},U(t))$ be defined as the Stokes problem corresponding to (\ref{the21}). With the initial value $\textbf{v}_{U(0)}(0)=\textbf{v}_0$, for $t\in[0,T=O(\varepsilon^{-\frac{1}{\alpha}})]$, it holds that
	\begin{equation}\label{the61}
		\begin{split}
			|u(t)-U(t)|\leq C_{T37a}\varepsilon^{\frac{1}{2}}, \quad \Vert \textbf{v}_{U(t)}(t)-\textbf{v}(t)\Vert_{L^2}\leq C_{T37b}\varepsilon^{\frac{\xi}{2}},
		\end{split}
	\end{equation}
	and 
	\begin{equation} \label{the300}
		\int_t^{t+1}\Vert\textbf{v}(s)-\textbf{v}_{\overline{U}(t)}(s)\Vert_{H^2} ds\leq C_{T37c}\varepsilon^{\frac{1}{2}}.
	\end{equation}
	$C_{T37a}$ and $C_{T37b}$ depend on $C_{L34b}$, $C_{L36a}$,  $C_{L36b}$ and the constants in Assumption \ref{assumption1}. $\xi$ follows the definition in \textbf{Lemma \ref{lemma3}}.
\end{theorem}

\begin{proof}
	Let $w(t)=U(t)-\overline{U}(t)=U(t)-u_0-(I^{\alpha}_{0^+}\int_{s}^{s+1}\frac{d^\alpha u(r)}{dr^\alpha}dr)(t)$, we have
	\begin{equation}\label{the62}
		\begin{split}
			&\frac{d^\alpha}{dt^\alpha}w(t)=\varepsilon\int_t^{t+1}R\left(\textbf{v}_{U(t)}(s),U(t)\right)ds-\varepsilon \int_t^{t+1}R\left(\textbf{v}(s),u(s)\right)ds\\\
			&=\varepsilon\int_t^{t+1}\left(R(\textbf{v}_{U(t)}(s),U(t))-R(\textbf{v}(s),u(s))\right)ds.
		\end{split}
	\end{equation}
	
	Using Assumption \ref{assumption1}, we obtain
	\begin{equation}\label{the63}
		\begin{split}
			&\left|\varepsilon\int_t^{t+1}\left(R(\textbf{v}_{U(t)}(s),U(t))-R(\textbf{v}(s),u(s))\right)ds\right|\\
			&\leq \varepsilon C_{A1b}\int_t^{t+1}|U(t)-u(s)|ds+\varepsilon C_{A1b}\int_t^{t+1}\Vert \textbf{v}_{U(t)}(s)-\textbf{v}(s)\Vert_{H^2} ds.
		\end{split}
	\end{equation}
	
	For the first term, using (\ref{the15}), we have
	\begin{equation}\label{the64}
		\begin{split}
			\int_t^{t+1}|U(t)-u(s)|ds&\leq \int_t^{t+1}\left|U(t)-\overline{U}(t)\right|ds+\int_t^{t+1}\left|u(s)-\overline{U}(t)\right|ds\\
			&\leq |w(t)|+C\varepsilon.
		\end{split}
	\end{equation}
	
	For the second term, combining \textbf{Lemma \ref{lemma2}}, \textbf{Lemma \ref{lemma3}} and Cauchy-Schwarz inequality, we have
	\begin{equation}\label{the65}
		\begin{split}
			&\int_t^{t+1}\Vert \textbf{v}_{U(t)}(s)-\textbf{v}(s)\Vert_{H^2} ds\\
			\leq& \int_t^{t+1}\Vert \textbf{v}_{U(t)}(s)-\textbf{v}_{u(s)}(s)\Vert_{H^2} ds+\int_t^{t+1}\Vert \textbf{v}_{u(s)}(s)-\textbf{v}(s)\Vert_{H^2} ds\\
			\leq& \int_t^{t+1}|U(t)-u(s)|ds+\left(\int_t^{t+1}\Vert \textbf{v}_{u(s)}(s)-\textbf{v}(s)\Vert_{H^2}^2 ds\right)^{\frac{1}{2}}\\
			\leq& |w(t)|+C\varepsilon+C\varepsilon^{\frac{1}{2}}\leq |w(t)|+C\varepsilon^{\frac{1}{2}}.
		\end{split}
	\end{equation}
	
	Taking $I^{\alpha}_{0^+}$ on both sides of (\ref{the62}), for $T=O(\varepsilon^{-\frac{1}{\alpha}})$, we obtain
	\begin{equation}\label{the66}
		\begin{split}
			|w(t)|\leq C\varepsilon^{\frac{1}{2}}+\frac{C\varepsilon}{\Gamma(\alpha)}\int_0^{t}(t-s)^{\alpha-1}|w(s)|ds,
		\end{split}
	\end{equation}
	
	Applying Gronwall's inequality, it is easy to obtain
	\begin{equation}\label{the67}
		\begin{split}
			|w(t)|\leq C\varepsilon^{\frac{1}{2}} e^{C\varepsilon t^\alpha}\leq C\varepsilon^{\frac{1}{2}}.
		\end{split}
	\end{equation}
	
	Using (\ref{the15}), \textbf{Lemma \ref{lemma2}} and \textbf{Lemma \ref{lemma3}}, we obtain
	\begin{equation}\label{the68}
		\begin{split}
			|U(t)-u(t)|\leq |w(t)|+\left|\overline{U}(t)-u(t)\right|\leq C\varepsilon^{\frac{1}{2}}+C\varepsilon\leq C_{T37a}\varepsilon^{\frac{1}{2}},
		\end{split}
	\end{equation}
	\begin{equation}
		\begin{split}
			&\Vert \textbf{v}_{U(t)}(t)-\textbf{v}(t)\Vert_{L^2}\leq \Vert \textbf{v}_{U(t)}(t)-\textbf{v}_{u(t)}(t)\Vert_{L^2}+\Vert \textbf{v}_{u(t)}(t)-\textbf{v}(t)\Vert_{L^2}\\
			&\leq C_{L34a}|U(t)-u(t)|+C_{L36a}^{1/2}\varepsilon^{\xi}\leq C_{T37b}\varepsilon^{\xi}
		\end{split}
	\end{equation}
	
	(\ref{the300}) is estimated by inserting $\textbf{v}_{u(t)}(t)$ and using \textbf{Lemma \ref{lemma2}} and \textbf{Lemma \ref{lemma3}}, we obtain
	\begin{equation}\label{the69}
		\begin{split}
			&\int_t^{t+1}\Vert\textbf{v}(s)-\textbf{v}_{\overline{U}(t)}(s)\Vert_{H^2} ds\leq \int_t^{t+1}\Vert\textbf{v}(s)-\textbf{v}_{u(t)}(s)\Vert_{H^2} ds\\
			&+\int_t^{t+1}\Vert\textbf{v}_{\overline{U}(t)}(s)-\textbf{v}_{u(t)}(s)\Vert_{H^2} ds\leq C\varepsilon^{\frac{1}{2}}+|w(t)|\leq C_{T37c}\varepsilon^{\frac{1}{2}}.
		\end{split}
	\end{equation}
	The proof is completed.
\end{proof}

\begin{remark}
	\label{remark3.8}
	In the integer order case ($\alpha=1$), it is obvious that $u'(t)=O(\varepsilon)$, the order of (\ref{the54}) and (\ref{the60}) are $O(\varepsilon^2)$ by using Young's inequality. Following the proof of \textbf{Theorem \ref{theorem4}}, we can easily obtain the order of (\ref{the61}) and (\ref{the300}) is $O(\varepsilon)$.
\end{remark}

\section{Numerical methods}
\label{section4}
In this section we introduce simple discrete schemes to approximate this two-way coupled multiscale problem. For the time discretization of fast scale, we use the first-order linear backward Euler scheme. The finite element method is used for spatial discretization, the velocity and pressure are approximated by $\mathcal{P}2-\mathcal{P}1$ elements. For the slow scale, the discretization is based on the $\mathcal{L}1$ scheme (See \cite{Gao2012}). We can easily choose a higher-order scheme for discretization, but the computational cost of long-term simulation are prohibitive.
\subsection{Numerical algorithms}
\label{section4.1}
We split the time interval $[0,T]$ into subintervals of equal size. Define $t_i=i\triangle t \ (i=0,1,2,...,M)$, where $\triangle t=T/M$ is the micro-scale time step. Define $T_j=j\triangle T \ (j=0,1,2,...,N)$, $\triangle T=T/N$ is the macro-scale time step. 

\subsubsection{Direct method}
For fractional differential equations, writes the $\mathcal{L}1$ scheme \cite{Gao2012, Stynes2021survey, Oldham1974fractional} to discretize the Caputo derivative of order $0<\alpha<1$
\begin{equation}\label{the72}
	\begin{split}
		\frac{d^\alpha}{dt^\alpha}u(t_i)=\frac{(\triangle t)^{-\alpha}}{\Gamma(2-\alpha)}\left[a_0u_i-\sum\limits_{j=1}^{i-1}(a_{i-j-1}-a_{i-j})u_j-a_{i-1}u_0\right]+K(u(t_i)),
	\end{split}
\end{equation}
where $a_j=(j+1)^{1-\alpha}-j^{1-\alpha}$ $(j\geq0)$ and $|K(u(t_i))|\leq C_{\alpha}\mathop{max}\limits_{t_0\leq t\leq t_M}|u''(t)|(\Delta t)^{2-\alpha}$.\\

The time discretization scheme of (\ref{the1}) is then as follows
\begin{equation}\label{the73}
	\begin{split}
		&\text{div} \ \textbf{v}_{i}=0, \quad \rho\left(\frac{\textbf{v}_i-\textbf{v}_{i-1}}{\triangle t}+(\textbf{v}_{i-1}\cdot\nabla)\textbf{v}_i\right)=\text{div} \  \sigma(\textbf{v}_i,p_i)+\textbf{f}_i, \quad in \ \Omega(u(t_i))\\
		&\frac{(\triangle t)^{-\alpha}}{\Gamma(2-\alpha)}\left[a_0u_i-\sum\limits_{k=1}^{i-1}(a_{i-k-1}-a_{i-k})u_k-a_{i-1}u_0\right]=\varepsilon R(\textbf{v}_{i-1},u_{i-1}),\\
		&\textbf{v}(0)=\textbf{v}_0, \quad u(0)=u_0.
	\end{split}
\end{equation}
The $\mathcal{L}1$ scheme is an implicit method. For simplicity, we change it to an explicit method and the order is $O(\Delta t)$.

The direct method is used for the forward simulation, and the algorithm is as follows.
\begin{algorithm}[H]
	\caption{Algorithm for the direct method.}
	\label{Algorithm1}
	\begin{algorithmic}
		\STATE{Let $\textbf{v}(0)=\textbf{v}_0$ and $u(0)=u_0$.}
		\FOR{$i=1:M$}
		\STATE{\textbf{Step 1.} Solve the third equation of $(\ref{the73})$
			to obtain $u_i$.}
		\STATE{\textbf{Step 2.} Based on $u_i$ obtained in \textbf{Step 1}, a fixed domain is generated and meshed.}
		\STATE{\textbf{Step 3.} Solve the coupled Navier-Stokes equation (the first and the second equation of $(\ref{the73})$ ) to obtain $\textbf{v}_i$.}
		\ENDFOR
	\end{algorithmic}
\end{algorithm}

\subsubsection{Multiscale method}
For the auxiliary time periodic problem (\ref{the21}), we propose the semi-discretization scheme of $U$
\begin{equation}\label{the160}
	\begin{split}
		&\frac{(\triangle T)^{-\alpha}}{\Gamma(2-\alpha)}\left[a_0U_j-\sum\limits_{k=1}^{j-1}(a_{j-k-1}-a_{j-k})U_k-a_{j-1}U_0\right]=\varepsilon \sum\limits_{i=1}^{M/T}\Delta t R((\textbf{v}_{U_{j-1}})_i,U_{j-1}),\\
		&U_0=u_0.
	\end{split}
\end{equation}

The semi-discretization scheme of the time periodic Navier-Stokes equations is as follows
\begin{equation}\label{the74}
	\begin{split}
		&\text{div} \ (\textbf{v}_{U_j})_i=0\\ &\rho\left(\frac{(\textbf{v}_{U_j})_i-(\textbf{v}_{U_j})_{i-1}}{\triangle t}+((\textbf{v}_{U_j})_{i-1}\cdot\nabla)(\textbf{v}_{U_j})_i\right)=\text{div} \ \sigma((\textbf{v}_{U_j})_i,(p_{U_j})_i)+\textbf{f}_i\\
		&\textbf{v}_{U_j}(0)=\textbf{v}_{U_j}(1), \quad in \ \Omega(U_j).
	\end{split}
\end{equation}

To solve the multiscale problem (\ref{the21}), it is necessary to identify the time-periodic solution. We first prove a lemma which provides an iterative method to find the time-periodic solution. Although the lemma is for the continuous Navier-Stokes equations, it can be proved in a very similar way for the temporal discrete scheme of the Navier-Stokes equations. Our numerical experiments demonstrate that the iterative method based on this lemma works very well in finding the initial value of the time-periodic problem of the Navier-Stokes equations.

\begin{lemma}
	\label{lemma4}
	For a fixed $U$ (flow domain is fixed), let $\textbf{v}^{U}(t)$ be the solution of the initial value problem and $\textbf{v}_U(t)$ be the time-periodic solution. Let $\textbf{v}^U_0=\textbf{v}_0$ be the initial trial value. We have the following result
	\begin{equation}\label{the75}
		\begin{split}
			\Vert \textbf{v}^U(n)-\textbf{v}_U(0)\Vert^2_{L^2}\leq (e^{-C})^{\frac{n(n+1)}{2}}\Vert \textbf{v}_0-\textbf{v}_U(0)\Vert^2_{L^2}
			\rightarrow 0 \ (n\rightarrow\infty).
		\end{split}
	\end{equation}
	In other words the solution of the initial value problem fastly approaches the initial value of the periodic problem under the uniqueness assumption (Assumption \ref{assumption2}).
	\begin{proof}
		For a fixed $U$, let $\textbf{w}(t)=\textbf{v}^U(t)-\textbf{v}_U(t)$. We have the following governing equations
		\begin{equation}\label{the76}
			\begin{split}
				&\textrm{div} \ \textbf{w}=0, \ \rho\frac{\partial\textbf{w}}{\partial t}+\rho\left((\textbf{v}^{U}\cdot\nabla)\textbf{v}^U-(\textbf{v}_{U}\cdot\nabla)\textbf{v}_U\right)=\rho\nu\Delta \textbf{w}-\nabla p^{U}+\nabla p_{U}, \ in \ \varOmega(U)\\
				&\textbf{w}(0)=\textbf{v}_0-\textbf{v}_U(0),\quad \textbf{w}=0 \ on \ \partial\varOmega(U).
			\end{split}
		\end{equation}
		
		Multiplying the second equation of (\ref{the76}) by $\textbf{w}$ and integrating with respect to the space variables on the domain $\varOmega(U)$, using Poincar\'{e} inequality, we obtain 
		\begin{equation}\label{the77}
			\begin{split}
				\frac{d}{dt}\left\Vert \textbf{w}(t)\right\Vert^2_{L^2}+C\Vert  \textbf{w}(t)\Vert^2_{L^2}\leq 0.
			\end{split}
		\end{equation}
		
		Multiplying (\ref{the77}) by $e^{Ct}$, then integrating from $t-1$ to $t$, we have 
		\begin{equation}\label{th78}
			\begin{split}
				\Vert \textbf{v}^U(t)-\textbf{v}_U(t)\Vert^2_{L^2}\leq e^{-Ct}\Vert \textbf{v}^U(t-1)-\textbf{v}_U(t-1)\Vert^2_{L^2}.
			\end{split}
		\end{equation}
		
		Let $t=1,2...n$, we have
		\begin{equation}\label{the79}
			\begin{split}
				\Vert \textbf{v}^U(n)-\textbf{v}_U(n)\Vert^2_{L^2}=\Vert \textbf{v}^U(n)-\textbf{v}_U(0)\Vert^2_{L^2}\leq e^{-nC}\Vert \textbf{v}^U({n-1})-\textbf{v}_U(0)\Vert^2_{L^2}.
			\end{split}
		\end{equation}
		
		For a constant $0<e^{-C}<1$, we obtain
		\begin{equation}\label{the80}
			\begin{split}
				\Vert \textbf{v}^U(n)-\textbf{v}_U(0)\Vert^2_{L^2}\leq (e^{-C})^{\frac{n(n+1)}{2}}\Vert \textbf{v}_0-\textbf{v}_U(0)\Vert^2_{L^2} \rightarrow 0 \ (n\rightarrow \infty).
			\end{split}
		\end{equation}
	\end{proof}
\end{lemma}

Based on \textbf{Lemma \ref{lemma4}}, we can give an initial trial value (usually the inflow velocity) and the tolerance $\tau>0$ to reach the initial value (and the end value) of the periodic solution, and then perform the calculation. In this process, we calculate the error every 1 second until the given tolerance is reached. The algorithm is as follows
\begin{algorithm}[H]
	\caption{Algorithm for the identification of the initial value of time-periodic solutions.}
	\label{Algorithm2}
	\begin{algorithmic}
		\STATE{Given an inflow velocity $\textbf{v}_0$, let $\tau>0$ be a given tolerance and let $n=0$. }
		\STATE{\textbf{Step 1.} Solve Navier-Stokes equations $(\ref{the74})$  with $U_j$ to obtain $\textbf{v}^{U_j}(n)=\textbf{v}^{U_j}\left(\frac{i\triangle tM}{T}\right)$.}
		\STATE{\textbf{Step 2.} Calculate the error\\
			
			$\epsilon:=\Vert \textbf{v}^{U_j}(n)-\textbf{v}^{U_j}(n-1)\Vert^2_{L^2(\varOmega(U_j))}=\Vert \textbf{v}^{U_j}(\frac{i\triangle tM}{T})-\textbf{v}^{U_j}(\frac{i\triangle tM}{T}-1)\Vert^2_{L^2(\varOmega(U_j))}$.}
		\STATE{\textbf{Step 3.} If $\epsilon< \tau$, stop; else, Update $n=n+1$, go to \textbf{Step 1}.}
		\STATE{\textbf{Step 4.} Output the last periodic part $\textbf{v}^{U_j}(\frac{i\triangle tM}{T}-1),...,\textbf{v}^{U_j}(\frac{i\triangle tM}{T})$.}
	\end{algorithmic}
\end{algorithm}

Besides the fast convergence we would like to point out that Algorithm 4.2 also makes it easy or effective to implement the solver of Navier-Stokes with the software COMSOL Multiphysics.

Now, the fast and slow scales of (\ref{the74}) have been separated. Having Algorithm \ref{Algorithm2} to find the initial value for the time-periodic problem $(\ref{the74})$, we introduce the following multiscale algorithm.
\begin{algorithm}[H]
	\caption{Fractional multiscale algorithm.}
	\label{Algorithm3}
	\begin{algorithmic}
		\STATE{Let $U_0=u_0$ and $j=1,2,...N$. Input $j=1$.}
		\STATE{\textbf{Step 1.} For a fixed $U_{j-1}$, solve the time periodic auxiliary problem (\ref{the74}) to obtain $(\textbf{v}_{U_{j-1}})_i$ based on Algorithm \ref{Algorithm2}.}
		\STATE{\textbf{Step 2.} Calculate the integral reaction term in equation $(\ref{the160})$ 
			$$R_{j-1}=\frac{T}{M}\sum\limits_{i=1}^{(M/T)}R\left(U_{j-1},(\textbf{v}_{U_{j-1}})_i \right)$$.}
		\STATE{\textbf{Step 3.} Step forward $U_{j-1}\rightarrow U_{j}$ and go to \textbf{Step 1} with the $\mathcal{L}1$ approximation
			$$U_{j}=\Gamma(2-\alpha)(\triangle T)^{\alpha}\varepsilon R_{j-1}+\sum\limits_{k=1}^{j-1}(a_{j-k-1}-a_{j-k})U_{k}+a_{j-1}U_0$$.}
	\end{algorithmic}
\end{algorithm}

\subsection{Implementations}
In the ALE formulation, the moving boundary may be transfered into a fixed boundary and then $u(t)$ which defines the moving boundary enters the flow equations, which are usually much more complex than the original flow equations (See, for example, (\ref{the28}) after the ALE transformation). In \cite{Frei2020} their numerical algorithm is based on the ALE and these ALE transferred flow equations are further simplified into quasilinear ODEs in their theoretical analysis. As seen earlier, our algorithms are based on simple front-tracking and the flow equations remain as their original simple form. This plus the finite different scheme and the iterative algorithm to identify the initial value of the time-periodic problem (See Section \ref{section4.1}) makes it particularly easy to use existing software to implement our algorithms, for example, a combination of the finite element software COMSOL Multiphysics 5.6 and MATLAB. In our numerical experiments later, we will test examples of 2-D Navier-Stokes problem, we connect the finite element software COMSOL Multiphysics 5.6 and MATLAB 2016b to realize the coupled computation. The time discretization is performed in MATLAB, and the spatial discretization is handled in COMSOL by using adaptive mesh. Specifically, in the multiscale computation, we use MATLAB to solve the concentration $U$, and fix the flow domain with the $U$. Then we pass the domain information to COMSOL and use it to find the time-periodic solution $\textbf{v}_{U}$, and the obtained $\textbf{v}_U$ is returned to MATLAB to complete an implementation cycle. This is very effective for our numerical experiments of direct and multiscale algorithms, and saves time in programming. This procedure may also be directly applied to engineering applications. 

To verify the performance of the multiscale method, we compare its solution with that of the original initial value problem (\ref{the1}) (called direct solution). In doing this direct computation, we set the time step $\triangle t=1/20$ and use 210 spatial elements in the domain. For the multiscale method, the number of spatial elements remains the same, and we change the macro-scale time step $\Delta T$ and the micro-scale time step $\Delta t$ to test the efficiency and accuracy of the proposed scheme. We would like to point out that the mesh is locally refined at the deformed boundary. Compared with the model with integer derivatives, using the direct method to achieve long-term calculation for the model with fractional derivatives is almost impossible due to the nonlocality (or the integral in the time interval $[0,t]$), therefore the total computation time $T$ is usually not too large in our error comparison.

\subsection{Error Analysis of the semi-discrete scheme}
\label{Error Analysis of the semi-discrete scheme}
There are already a lot of error analysis available for numerical schemes of Navier-Stokes equation (See e.g. \cite{Geveci1989convergence, Ju2002global} and references therein), although very few are for the Navier-Stokes equations coupled with a fractional differential equation. In this section we analyse the error of the temporal semi-discrete scheme for the fractional part of the system, assuming that the temporal discrete error estimate has been done for the integer order Navier-Stokes part. In view of this, we make the following assumption.

\begin{assumption}
	\label{assumption4}
	We assume that $\tau>0$ is the tolerance in Algorithm \ref{Algorithm2} to solve the time periodic auxiliary problem. The error of the fast scale first-order linear backward Euler scheme can be expressed as follows, for all time step $t_i$,
	\begin{equation}\label{the81}
		\begin{split}
			\Vert \textbf{v}_U(t_i)-(\textbf{v}_U)_i\Vert_{H^1}+\left(\Delta t\sum\limits_{i=1}^{M/T}\Vert \textbf{v}_U(t_i)-(\textbf{v}_U)_i\Vert_{H^2}^2\right)^{\frac{1}{2}}\leq C_{A4}\Delta t+\tau,
		\end{split}
	\end{equation}
	where $C_{A4}$ is a constant.
\end{assumption}

\begin{remark}
	We can adapt the temporal error estimates in, for example, \cite[equations (3.32)-(3.40)]{Li2022temporal} for the initial value problem of Navier-Stokes equation, by assuming sufficient regularity and compatibility conditions and constant density. For the time-periodic problem the estimates may be obtained as well by additional applying the technique of initial value estimates (\ref{the38})-(\ref{the39}) to obtain the initial value error first.
\end{remark}

\begin{theorem}
	\label{theorem6}
	Let $u\in C^2[0,T]$ and $U\in C^2[0,T]$ be the solutions to the original problem (\ref{the1}) and the time averaged problem (\ref{the21}), respectively, $U_j$ be the solution of the discrete effective equations (\ref{the74}). For $T_j=O(\varepsilon^{-\frac{1}{\alpha}})$, we have the following error estimate
	\begin{equation}\label{the82}
		\begin{split}
			|u(T_j)-U_j|\leq C_{T44}(\varepsilon+\Delta t+\tau+\Delta T\varepsilon^{\frac{1}{\alpha}}),
		\end{split}
	\end{equation}
	where $C_{T44}$ is a positive constant.
\end{theorem}

\begin{proof}
	We first consider the regularity estimate of $U$. Based on the definition of the Caputo fractional derivative, by direct computation, we have
	\begin{equation}\label{the280}
		\begin{split}
			&|U'(t)|=|D_{0^+}^{1-\alpha}D_{0^+}^{\alpha}U(t)|\leq\frac{\varepsilon}{\Gamma(\alpha)}\int_{0}^t (t-s)^{\alpha-1}\left|d_s\int_{0}^1 R(\textbf{v}_{U(s)}(r),U(s))drds\right|\\
			&\leq C\varepsilon\int_{0}^t (t-s)^{\alpha-1}U'(s)ds=C\varepsilon\int_{0}^t (t-s)^{2\alpha-1}(t-s)^{-\alpha}U'(s)ds\\
			&\leq  C\varepsilon^2\cdot t^{2\alpha-1}= O(\varepsilon^{\frac{1}{\alpha}}).
		\end{split}
	\end{equation}
	Moreover, it is easy to see $U'(0)=0$, and $(d_t\int_{0}^1 R(\textbf{v}_{U(t)}(s),U(t))ds)\big|_{t=0}=0$.
	
	Therefore, we have the following estimate of $U''(t)$ by integrating by parts
	\begin{equation}
		\begin{split}
			&|U''(t)|=|d_{t}D_{0^+}^{1-\alpha}D_{0^+}^{\alpha}U(t)|=\left|\frac{\varepsilon}{\Gamma(\alpha)}\frac{d}{dt}\int_{0}^{t}(t-s)^{\alpha-1}d_s\int_{0}^1 R(\textbf{v}_{U(s)}(r),U(s))drds\right|\\
			&\leq \left|\frac{\varepsilon}{\Gamma(\alpha+1)}\int_{0}^{t}(t-s)^{\alpha-1}d_s^2\int_{0}^1 R(\textbf{v}_{U(s)}(r),U(s))drds\right|\\
			&\leq C\varepsilon\int_{0}^{t}(t-s)^{\alpha-1}|U''(s)|ds+C\varepsilon^\frac{2}{\alpha}.
		\end{split}
	\end{equation}
	
	Applying Gronwall inequality to the above inequality, we obtain $|U''(t)|= O(\varepsilon^{\frac{2}{\alpha}})$.
	
	Let $E_j=U(T_j)-U_j$. Applying Taylor expansions of $U(T_j)$ around $T_{j-1}$, \textbf{Lemma \ref{lemma2}}, we obtain the error equation
	\begin{equation}\label{the83}
		\begin{split}
			&\frac{(\triangle T)^{-\alpha}}{\Gamma(2-\alpha)}\left[a_0E_j-\sum_{k=1}^{j-1}(a_{j-k-1}-a_{j-k})E_k-a_{j-1}E_0\right]\\
			&=\varepsilon\Delta t\sum\limits_{i=1}^{M/T}\left(R(\textbf{v}_{U(T_{j-1})}(t_i),U(T_{j-1}))-R(\left(\textbf{v}_{U_{j-1}}\right)_i,U_{j-1})\right)\\
			&+C\varepsilon^{1+\frac{1}{\alpha}}\Delta T+C\varepsilon\Delta t+C\varepsilon^{\frac{2}{\alpha}}(\Delta T)^{2-\alpha}.
		\end{split}
	\end{equation}
	
	With the help of the Lipschitz condition of $R$, we have
	\begin{equation}\label{the84}
		\begin{split}
			&|E_j|\leq \sum_{k=1}^{j-1}(a_{j-k-1}-a_{j-k})|E_k|+\Gamma(2-\alpha)(\Delta T)^{\alpha}\Delta t\varepsilon\sum\limits_{i=1}^{M/T}\left\Vert \textbf{v}_{U(T_{j-1})}(t_i)-(\textbf{v}_{U_{j-1}})_i\right\Vert_{H^2}\\ &+\Gamma(2-\alpha)(\Delta T)^{\alpha}\varepsilon|E_{j-1}|+C(\Delta T)^{\alpha}\Delta t\varepsilon+C(\Delta T)^{1+\alpha}\varepsilon^{1+\frac{1}{\alpha}}+a_{j-1}|E_0|.
		\end{split}
	\end{equation}
	
	Combining \textbf{Lemma \ref{lemma2}}, Assumption \ref{assumption4} and Cauchy-Schwarz inequality, we obtain
	\begin{equation}\label{the85}
		\begin{split}
			&\Delta t\sum\limits_{i=1}^{M/T}\left\Vert \textbf{v}_{U(T_{j-1})}(t_i)-(\textbf{v}_{U_{j-1}})_i\right\Vert_{H^2}\leq \Delta t\sum\limits_{i=1}^{M/T}\left\Vert \textbf{v}_{U(T_{j-1})}(t_i)-\textbf{v}_{U_{j-1}}(t_i)\right\Vert_{H^2}\\
			&+\Delta t\sum\limits_{i=1}^{M/T}\left\Vert \textbf{v}_{U_{j-1}}(t_i)-(\textbf{v}_{U_{j-1}})_i\right\Vert_{H^2}\leq C|E_{j-1}|+C\Delta t+\tau.
		\end{split}
	\end{equation}
	
	(\ref{the84}), (\ref{the85}), and the inequality $a_{j-1}>(1-\alpha)j^{-\alpha}$ lead to the following inequality,
	\begin{equation}\label{the86}
		\begin{split}
			&|E_j|\leq \sum_{k=1}^{j-1}(a_{j-k-1}-a_{j-k})|E_k|+C(\Delta T)^\alpha\varepsilon |E_{j-1}|+a_{j-1}\left[|E_0|+\frac{C(\Delta T)^\alpha\varepsilon}{a_{j-1}}\left(\Delta t+\Delta T\varepsilon^{\frac{1}{\alpha}}+\tau\right)\right]\\
			&\leq \sum_{k=1}^{j-1}(a_{j-k-1}-a_{j-k})|E_k|+C(\Delta T)^\alpha\varepsilon |E_{j-1}|+a_{j-1}\left[|E_0|+C(j\Delta T)^\alpha\varepsilon\left(\Delta t+\Delta T\varepsilon^{\frac{1}{\alpha}}+\tau\right)\right].
		\end{split}
	\end{equation}
	
	Now, we will prove the following estimate by induction,
	\begin{equation}\label{the87}
		\begin{split}
			|E_j|\leq& C\left(|E_0|+\Delta t+\Delta T\varepsilon^{\frac{1}{\alpha}}+\tau\right).
		\end{split}
	\end{equation}
	
	For $j=1$, we can easily obtain $|E_1|\leq C\left(|E_0|+\Delta t+\Delta T\varepsilon^{\frac{1}{\alpha}}+\tau\right)$.
	
	Suppose now (\ref{the87}) holds for $1,2,...,j-1$, we need then to prove that it holds also for $j$. From (\ref{the86}), we obtain
	\begin{equation}\label{the89}
		\begin{split}
			|E_j|\leq& C\sum_{k=1}^{j-1}(a_{j-k-1}-a_{j-k})\left(|E_0|+\Delta t+\Delta T\varepsilon^{\frac{1}{\alpha}}+\tau\right)+C(\Delta T)^\alpha\varepsilon |E_{j-1}|\\
			&+a_{j-1}\left[|E_0|+C(T_j)^\alpha\varepsilon\left(\Delta t+\Delta T\varepsilon^{\frac{1}{\alpha}}+\tau\right)\right]\\
		\end{split}
	\end{equation}
	
	For $T_j=O(\varepsilon^{-\frac{1}{\alpha}})$, we have
	\begin{equation}\label{the90}
		\begin{split}
			|E_j|\leq& C\left(\sum_{k=1}^{j-1}(a_{j-k-1}-a_{j-k})+a_{j-1}\right)\left(|E_0|+\Delta t+\Delta T\varepsilon^{\frac{1}{\alpha}}+\tau\right)+C(\Delta T)^\alpha\varepsilon |E_{j-1}|\\
			\leq& Ca_0\left(|E_0|+\Delta t+\Delta T\varepsilon^{\frac{1}{\alpha}}+\tau\right)+C(\Delta T)^\alpha\varepsilon\left(|E_0|+\Delta t+\Delta T\varepsilon^{\frac{1}{\alpha}}+\tau\right)\\
			\leq& C\left(|E_0|+\Delta t+\Delta T\varepsilon^{\frac{1}{\alpha}}+\tau\right),
		\end{split}
	\end{equation}
	which completes the proof of the estimate (\ref{the87}). With $E_0=0$, we obtain 
	\begin{equation}
		|E_j|\leq C\left(\Delta t+\Delta T\varepsilon^{\frac{1}{\alpha}}+\tau\right).
	\end{equation}
	
	The result follows by \textbf{Theorem \ref{theorem4}}, we have
	\begin{equation}\label{the92}
		\begin{split}
			|u(T_j)-U_j|\leq& |u(T_j)-U(T_j)|+|U(T_j)-U_j|\leq C_{T37a}\varepsilon^{\frac{1}{2}}+C\left(\Delta t+\Delta T\varepsilon^{\frac{1}{\alpha}}+\tau\right)\\
			\leq& C_{T44}(\varepsilon^{\frac{1}{2}}+\Delta t+\tau+\Delta T\varepsilon^{\frac{1}{\alpha}}).
		\end{split}
	\end{equation}
\end{proof}

\begin{remark}
	Here we assume $\frac{1}{2}\leq\alpha<1$ in the theorem is only to avoid the weak singularity at $t=0$ in derivtives of $U(t)$ (see (\ref{the280})). If $0<\alpha<1/2$, we expect that there will generally be a weak singularity in derivatives of $U(t)$ near $t=0$ and the convergence order of the $\mathcal{L}1$ finite difference scheme will be affected there too. A further discussion on this is beyond the scope of this paper. We refer to \cite{Stynes2021survey} (Section 4) for a discussion on using a so-called graded temporal steps near $t=0$ to improve the convergence order near the weak singularity.
\end{remark}	

\section{Numerical experiments}
\label{section5}
In this section, we carry out two numerical experiments to test the accuracy and effectiveness of the proposed multiscale method. We first perform a simple numerical test with a system of ODEs as a simplified model, such that the exact solution can be obtained to verify the accuracy of the algorithm. Then we carry out numerical tests on the atherosclerotic problem with a plaque growth modeled by Navier-Stokes equations to show the advantages of the multiscale method. Finally, we evaluate the effect of the fractional parameter $\alpha$, which describes the growth rate of plaque.

\subsection{Test of ODEs system}
To obtain the exact solutions of $v(t)$ and $u(t)$, the construction of reaction term $R$ is very simple without using the form in (\ref{the3}).
\begin{equation}\label{the93}
	\begin{split}
		&\frac{d}{dt}v(t)+u(t)v(t)=2\pi \cos(2\pi t)+u(\sin(2\pi t)+2), \quad \frac{d^{0.8}}{dt^{0.8}}u(t)=\varepsilon v(t)=5\cdot 10^{-4}v(t)\\
		&v(0)=2,\quad u(0)=1,\quad t\in [0,14001\approx \varepsilon^{-\frac{1}{\alpha}}].
	\end{split}
\end{equation}

This fractional problem has the exact solution: \\
$v(t)=sin(2\pi t)+2$, and  $u(t)=1+\varepsilon\left(\frac{2t^{\alpha}}{\Gamma(\alpha+1)}+2\pi t^{\alpha+1}E_{2,\alpha+2}(-4\pi^2 t^2)\right)$. $E_{\mu,\nu}(z)=\sum\limits_{k=0}^{\infty}\frac{z^{k}}{\Gamma(\mu k+\nu)}$ is the generalized Mittag–Leffler (M-L) function \cite{prabhakar1971singular}.

We set the macro step size $\Delta T=1000$, micro step size $\Delta T=1/100$ and trial value $(v_{U_j})_0=0.5$ in the simulation. It can be seen from Figure \ref{compare1} that the numerical results of the multiscale method agree well with the exact solution. This also verifies the accuracy of Algorithm \ref{Algorithm1}-\ref{Algorithm3}. The micro equation is discretized using the standard first order backward Euler scheme, and we further choose different macro step size to test the convergence rate of the discrete scheme (\ref{the160}) of the macro equation. Quantitative error measurements and convergence rates are shown in Table \ref{table1}. It is observed that our numerical method is first-order convergence for macro variable, which is consistent with the results of numerical analysis.

\begin{figure}[H]
	\centering
	\includegraphics[scale=0.47]{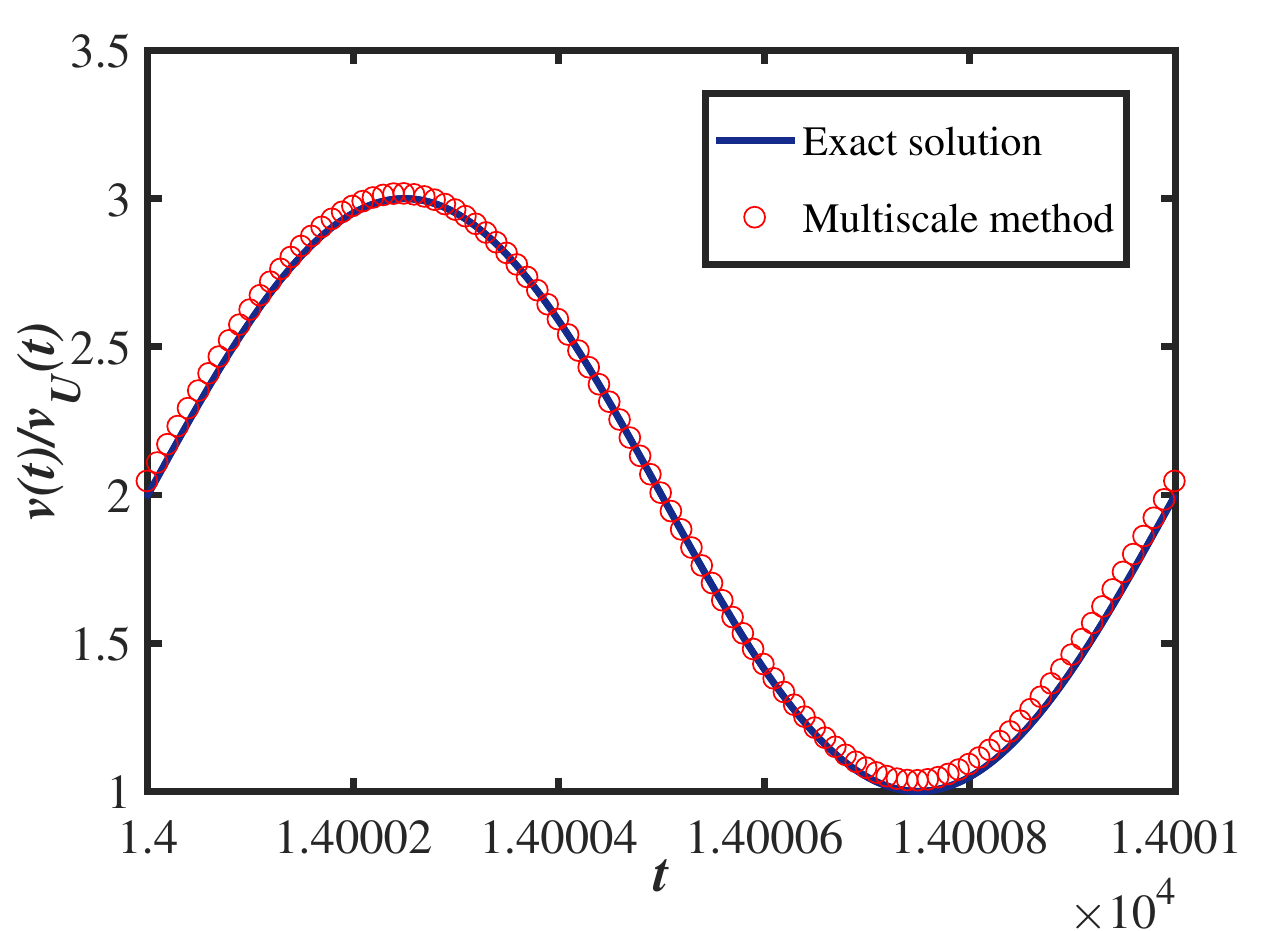}
	\includegraphics[scale=0.47]{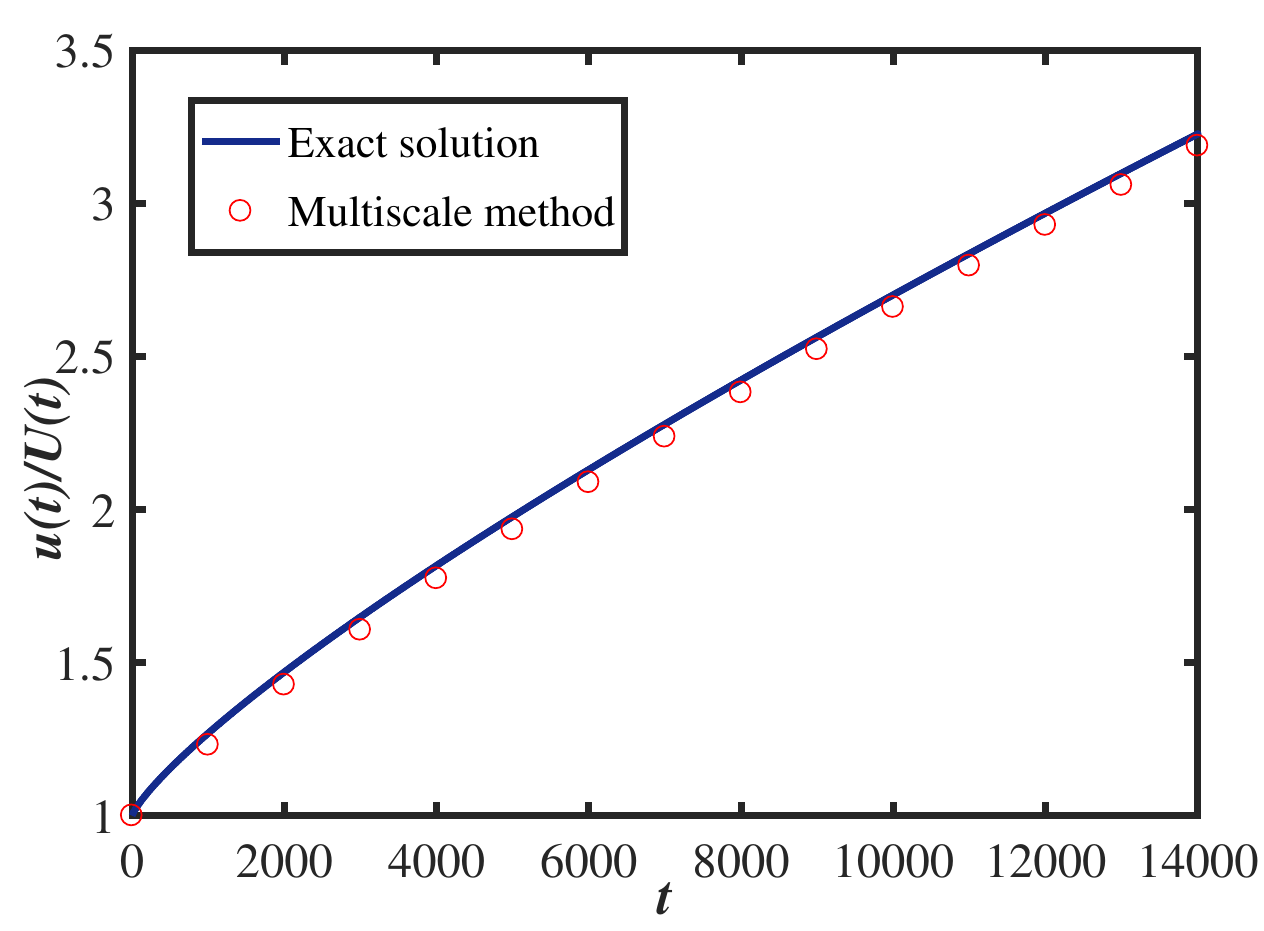}
	\caption{Comparison of the numerical solution and exact solution with $\triangle t=1/100$ and $\triangle T=1000$.}
	\label{compare1}
\end{figure}

\begin{table}[H]
	{\footnotesize
		\caption{Error of macro variable $U$ at $t=14000$ s, convergence rates and CPU time with fixed $\Delta t=1/100$ and $\Delta T\varepsilon^{\frac{1}{\alpha}}\to\Delta t$.}
		\label{table1}
		\begin{center}
			\begin{tabular}{|c|c|c|c|} \hline  		
				$\triangle T$ & Error & Order & CPU time \\ \hline
				2000 & 7.320e-2 & - & 0.32 s \\
				1000 & 3.941e-2 & 0.90 & 0.60 s \\  
				500 & 2.041e-2 & 0.95 & 1.13 s \\  
				250 & 1.004e-2  & 1.02  & 1.64 s \\  \hline
			\end{tabular}
		\end{center}
	}
\end{table}

\subsection{Test of 2-D flow problem}
In this subsection, we carry out numerical tests for Navier-Stokes flow. The results obtained by the direct method are used as a reference to verify the effect of the multiscale method and give the convergence rates.

We test the full Navier-Stokes system in (\ref{the1}). The flow domain is shown in Figure \ref{Pic1} with the shape function $\gamma(u,x)=ue^{-x^2}$.
\begin{equation}\nonumber
	\begin{split}
		\varOmega(u)=\left\{(x,y): |x|< 5 \ \text{cm}, \ -2 \ \text{cm}<y< (2-ue^{-x^2}) \ \text{cm}\right\}.
	\end{split}
\end{equation}

The time periodic Dirichlet condition is set on the left inflow boundary, which appears to make the problem periodic. We set $\textbf{v}_{in}=30\left(1-\frac{y^2}{4}\right)\sin^2(\pi t)$ cm/s. On the right outflow boundary we set a frequently used pressure condition: $-p\vec{n}+\rho\nu\frac{\partial \textbf{v}}{\partial \vec{n}}=0$. The parameters in the fluid and reaction term are the same as in \cite{Frei2020} which are claimed to mimic the real fluid-structure interaction problem: $\rho=1 \ \text{g}/\text{cm}^3$, $\nu=0.04 \ \text{cm}^2/\text{s}$ and $\sigma_0=30 \ \text{g}/(\text{cm} \ \text{s}^2)$. Moreover, we set fractional parameter $\alpha=0.6$, $\varepsilon=8\cdot10^{-4} \ {\text{s}^{-1}}$, computation time $T=8000 \ \text{s}$ and $u(0)=u_0=0.2$. 

For the multiscale method, we set the tolerance for the time periodic solution as $\Vert\textbf{v}_U(t+1)-\textbf{v}_U(t)\Vert^2_{L^2(\varOmega(U))}\leq\tau = 10^{-6}$. We choose different macro step sizes for comparison with the direct method and define Error$=|u_j-U_j|$. The errors and CPU time are shown in Figure 6. Even for not so large computing time $T=8000$ s, the CPU time to solve such a fluid-structure interaction problem using the direct method is 77 hours, which demonstrates the excellent performance of our multiscale method. 

We test the convergence rate of Scheme (\ref{the160})-(\ref{the74}). The result of numerical analysis (\ref{the82}) shows that the effect of the macro step size $\Delta T$ is dominant when $\Delta T>\varepsilon^{-1}\Delta t$, and vice versa. We fix the micro step size $\Delta t=1/40$ and macro step size $\Delta T=80$ to test the convergence order of $\Delta T$ and $\Delta t$, respectively. We then refine the macro time step or micro time step by 2 and calculate the error at $T=8000$ s. It is observed from Figure \ref{rate} that our numerical method achieves the expected first-order convergence for $\Delta T$ and $\Delta t$.

It should be noted that the direct long-term fractional calculation poses challenge to the computer capacity (in terms of CPU time and RAM) since the solution has nonlocal dependence on previous steps. The multiscale method constructed and theoretically justified in this paper significantly reduces the computational cost associated with the micro-scale steps needed in the direct computation and can thus solve the problem very effectively.

\begin{figure}[H]
	\centering
	\includegraphics[scale=0.38]{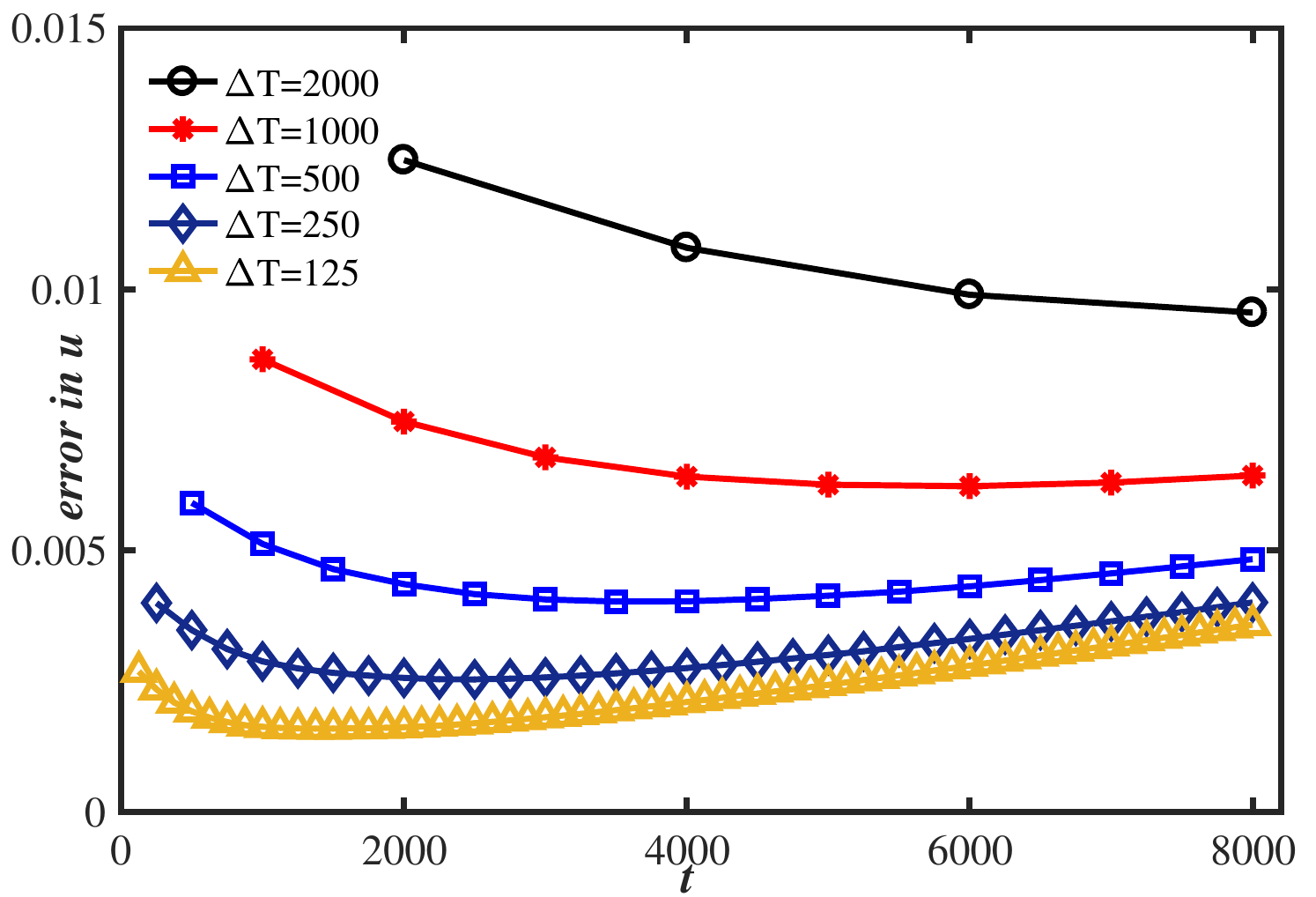}
	\includegraphics[scale=0.40]{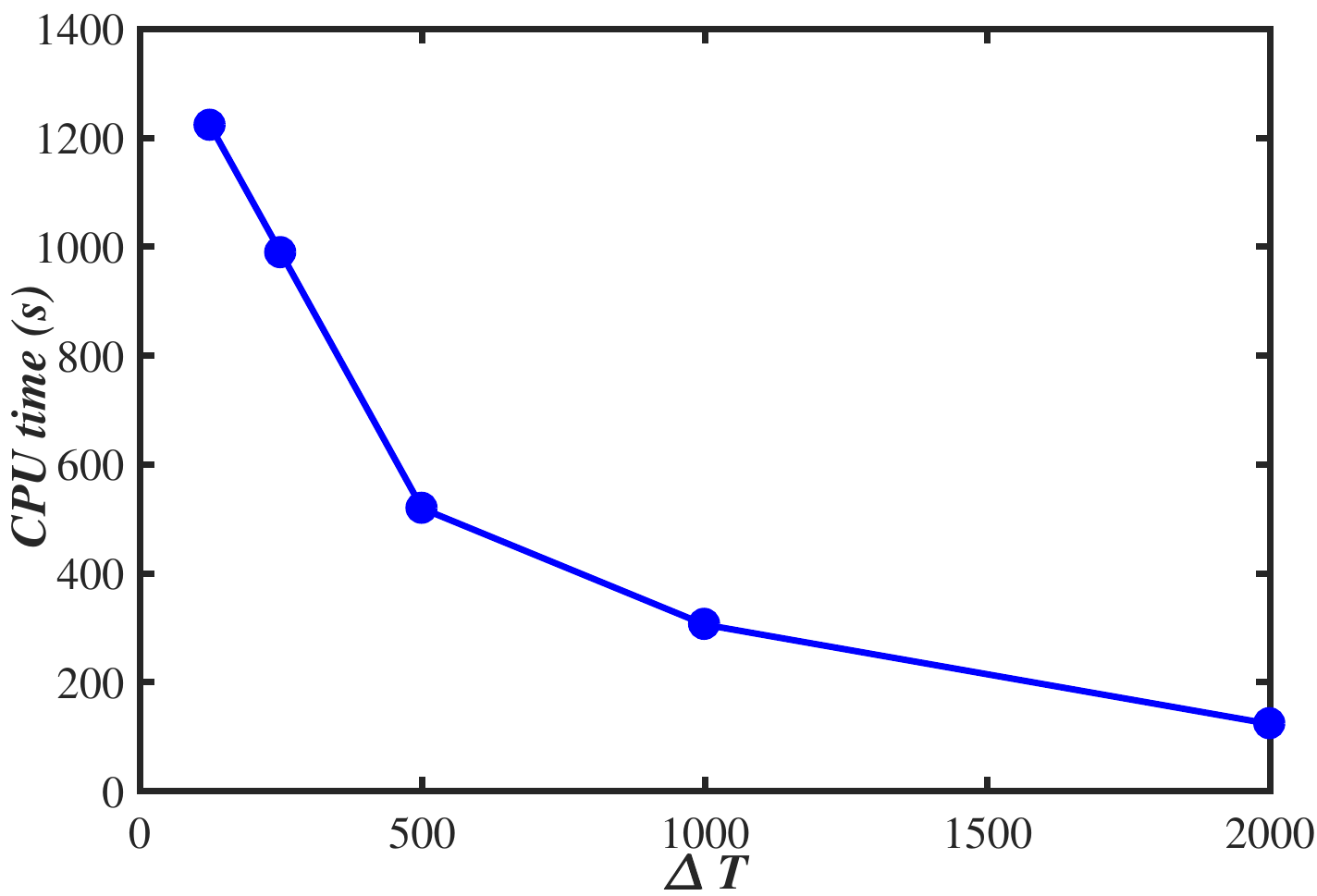}
	\caption{Error result (left) and CPU time (right)of different macro step size $\triangle T$ with fixed $\triangle t=1/20$ s.}
	\label{compare4}
\end{figure}

\begin{figure}[H]
	\centering
	\includegraphics[scale=0.46]{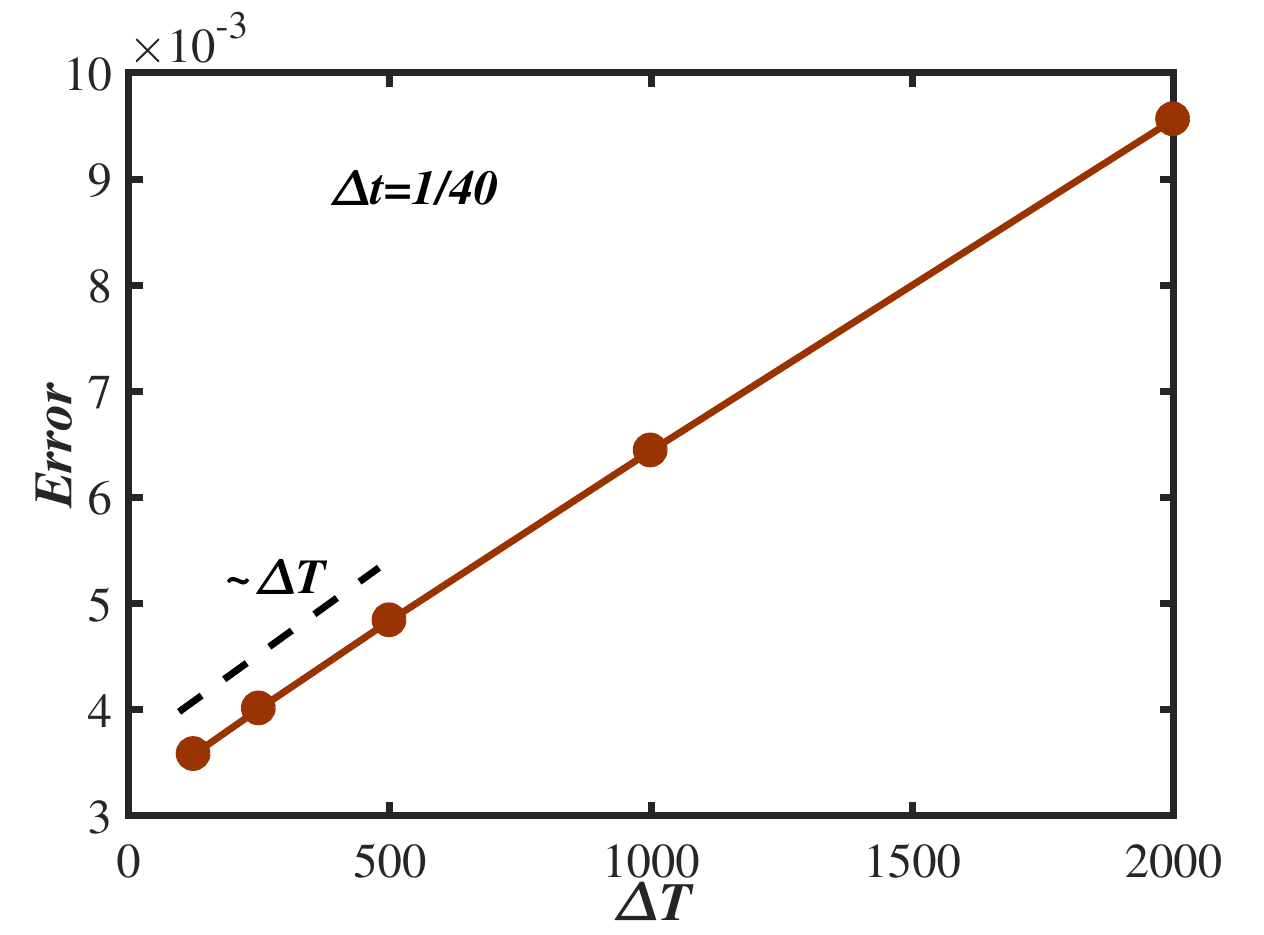}
	\includegraphics[scale=0.46]{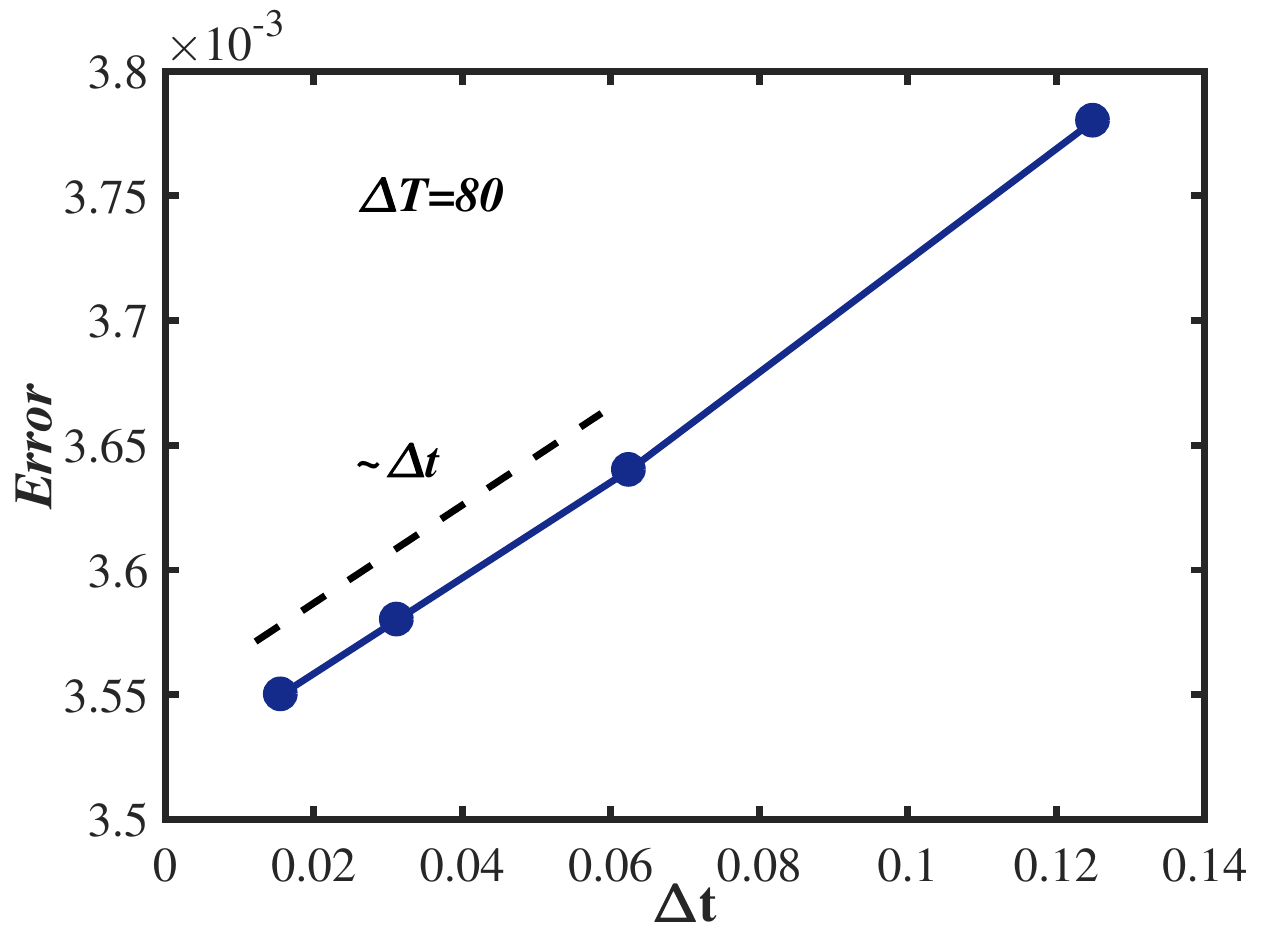}
	\caption{Convergence rate of the multiscale method at time $T=8000$ s.}
	\label{rate}
\end{figure}

\subsection{The effect of fractional order parameter on plaque growth}
Finally, we consider the effect of the fractional parameter $\alpha$ on the atherosclerotic plaque growth and hemodynamics. Here, the vessel width is proportionally enlarged to 3 cm, and the variable domain is given by 
\begin{equation}\nonumber
	\begin{split}
		\varOmega(u)=\left\{(x,y): |x|< 5 \text{cm}, \ -1.5 \ \text{cm}<y< (1.5-ue^{-x^2}) \ \text{cm}\right\}.
	\end{split}
\end{equation}
The parameters in the simulation are fixed as $T=1.8\cdot 10^{6}\ \text{s}\approx 21\ \text{days}$, $\rho=1 \ \text{g}/\text{cm}^3$, $\nu=0.04 \ \text{cm}^2/\text{s}$, $\sigma_0=30 \ \text{g}/(\text{cm} \ \text{s}^2)$, $\varepsilon=2\cdot 10^{-6} \ \text{cm}/\text{s}$. The inflow boundary condition is $\textbf{v}_{in}=20\left(1-\frac{y^2}{1.5^2}\right)\sin^2(\pi t) \ \text{cm/s}$, and the outflow boundary is the pressure boundary condition. Assume that the initial concentration is $u_0=0$ (without plaque formation). We set the macro- and micro-scale temporal step sizes to be $\triangle T=4\cdot10^{4} \ $s and $\triangle t=1/20\ $s, respectively.

The velocity magnitude and plaque growth when the fractional order parameter $\alpha=0.85$ and $\alpha=0.95$ are shown in Figure \ref{compare5}, with a strong narrowing of the flow domain. We expect that the computational cost would be huge for such a large $T$ by using the direct method since the fractional derivative is involved with an integral from $0$ to $T$. For a long-term simulation, the computational cost of the multiscale method is significantly reduced. Figure \ref{compare6} shows the effect of different fractional order parameter on plaque growth. The concentration of macrophages in the vessel wall increases with the increase of $\alpha$, which leads the accumulation of foam cells and plaque formation. Therefore, $\alpha$ can be regarded as a parameter to describe the plaque growth rate, a key index in the study of atherosclerosis.

\begin{figure}[H]
	\centering
	\includegraphics[scale=0.27]{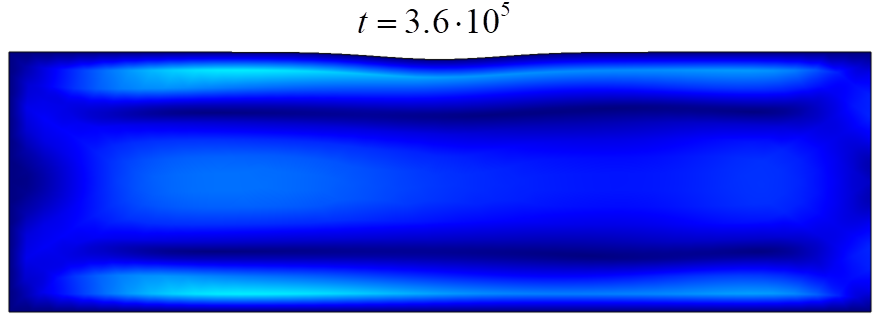}
	\includegraphics[scale=0.27]{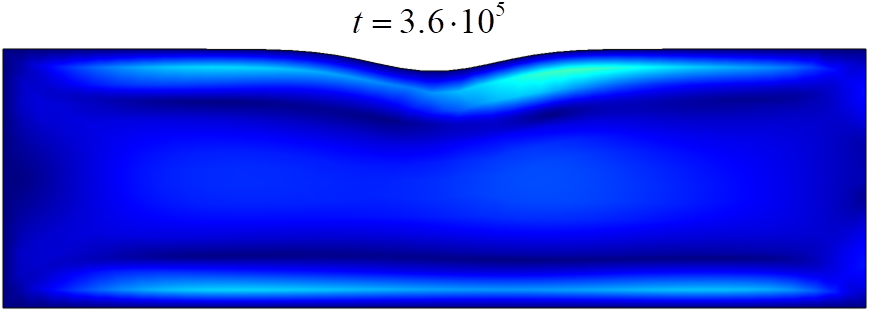}
	\quad
	\\
	\includegraphics[scale=0.27]{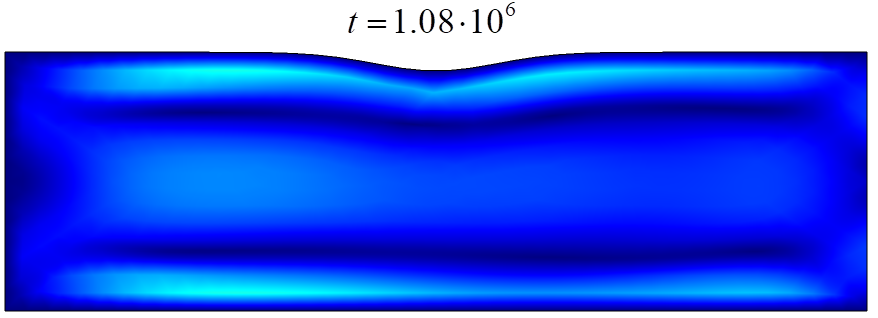}
	\includegraphics[scale=0.27]{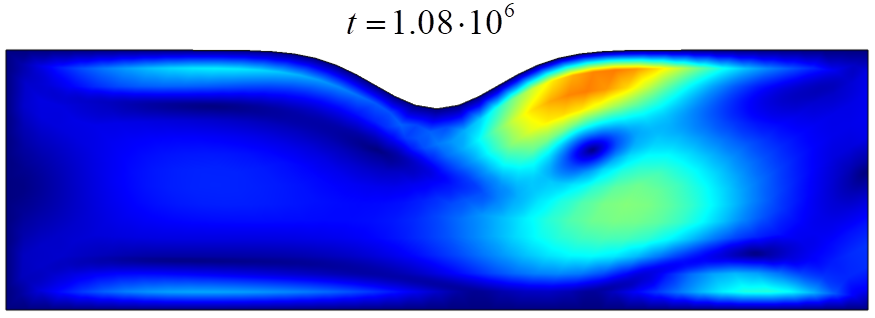}
	\quad
	\\
	\includegraphics[scale=0.27]{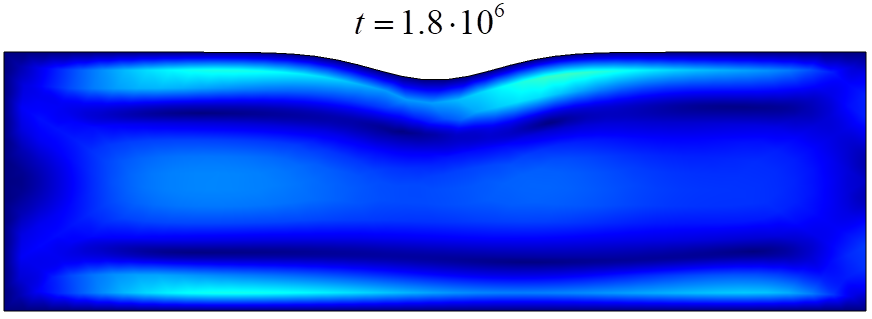}
	\includegraphics[scale=0.27]{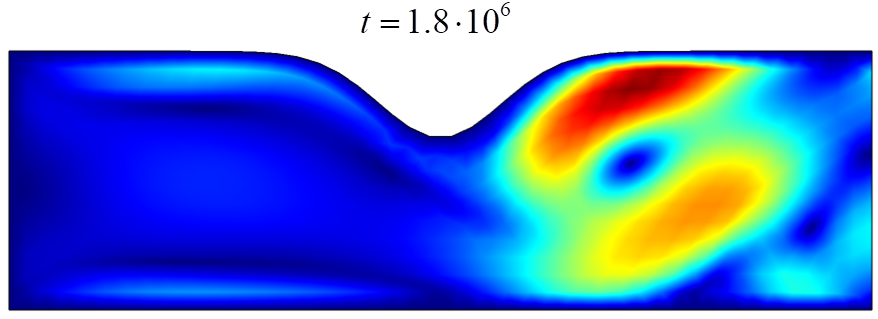}
	\caption{Snapshots of velocity field and plaque growth with $\alpha=0.85$ (left) and $\alpha=0.95$ (right).}
	\label{compare5}
\end{figure}

\begin{figure}[H]
	\centering
	\includegraphics[scale=0.45]{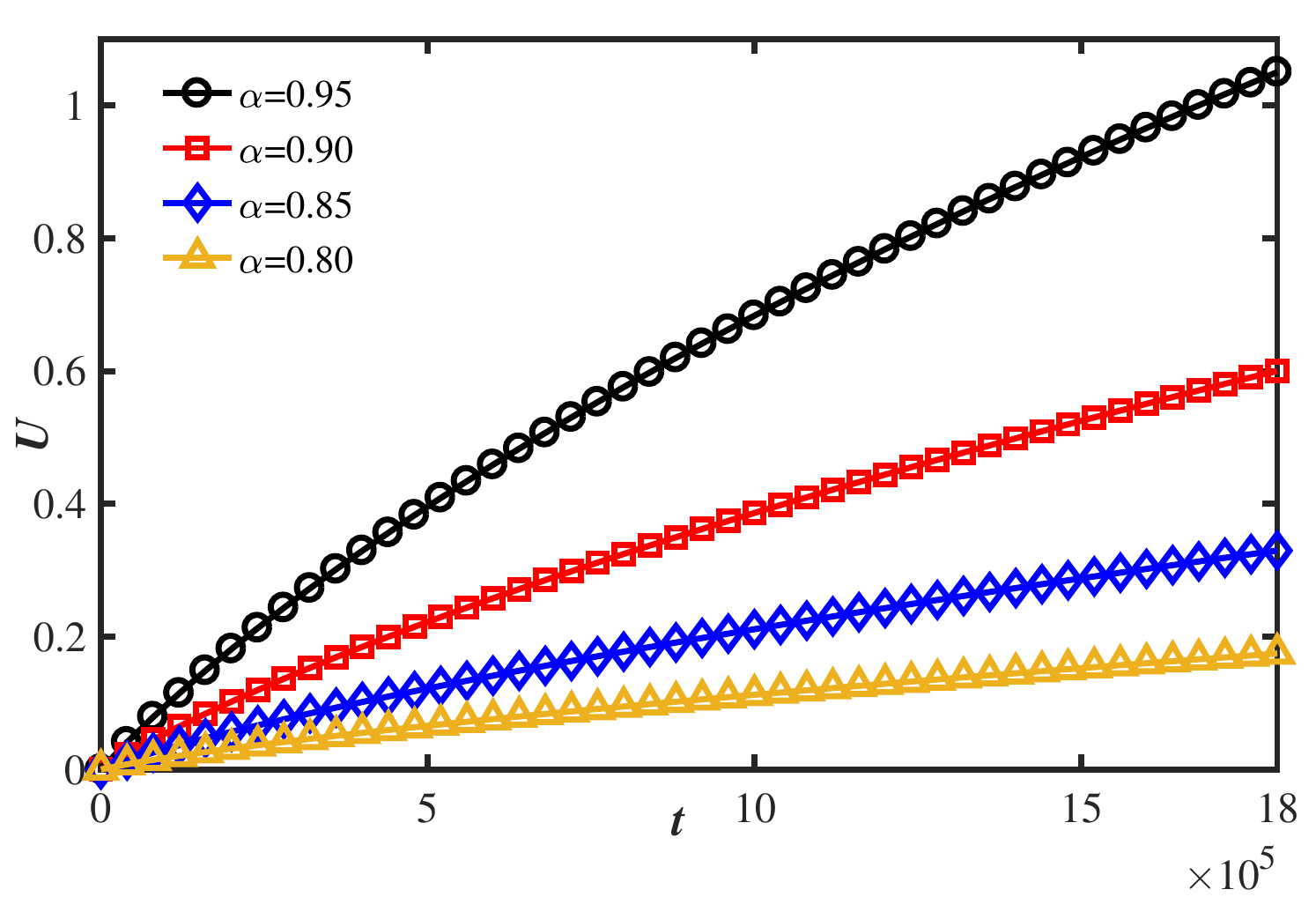}
	\caption{Effects of different fractional parameter $\alpha$ on plaque growth.}
	\label{compare6}
\end{figure}

\section{Conclusions and remarks}
\label{section6}
In this paper, we study a fluid-structure interaction problem with periodic forcing and temporal multiscale features (fast flow/slow plaque growth) and the slow variable equation contains fractional derivatives where a memory effect of the plaque evolution is included. We use the simple front tracking method to deal with the boundary growth and then formulate an auxiliary time periodic flow equation. Under the front tracking framework and for the linear Stokes flow, we analyze the error between solutions of the auxiliary periodic problem and the original problem. The error estimate is also expected to hold for the full Navier-Stokes equation. Based on this auxiliary time periodic problem, we then design an efficient multiscale algorithm and implement it in a combination of COMSOL Multiphysics and MATLAB. An iterative procedure to solve the time-periodic problem is designed and its exponential convergence is analyzed. An error analysis for a fractional order time-discrete scheme of the multiscale algorithm is also provided. Several numerical experiments are conducted to illustrate the accuracy and efficiency of the proposed multiscale algorithm. The test results for  simplified ODE systems and coupled Navier-Stokes systems show the great performance of the algorithm and that refining the macro-scale time step size can reduce the error more significantly than refining the micro-scale time step size. The final numerical example of the plaque growth problem shows the effect of different fractional order parameter on the plaque growth, and illustrates that the fractional order parameter may be used as an alternative index to reflect the plaque growth rate. 

For future work, we shall consider to include the reaction-diffusion equation into the plaque growth model (PDE/PDE system) and explore the possibility of designing and analysing an effective multiscale method.

\section*{Acknowledgments}
PL thanks Professor Thomas Richter for bringing his attention to the temporal multiscale fluid-structure interaction problem. This work is supported by the National Natural Science Foundation of China (Nos. 11861131004, 11771040, 11871339) and the Fundamental Research Funds for the Central Universities (No. 06500073).

\bibliographystyle{siamplain}
\bibliography{Ref}

\end{document}